\documentclass[12pt]{article}

\usepackage{latexsym}
\usepackage{amsxtra}
\usepackage{amssymb}
\usepackage{bm}
\usepackage{amsthm}
\usepackage{mathtools}
\usepackage{xspace}
\usepackage{graphicx}
\usepackage[margin=1in]{geometry}
\usepackage{xfrac}

\input xy
\xyoption{all} \CompileMatrices

\begin{document}

\def\HH{{\mathcal{H}}}
\def\orb{{\operatorname{orb}}}
\def\diam{{\operatorname{diam}}}
\def\II{{\mathfrak{I}}}
\def\PO{{\operatorname{PO}}}
\def\Cl{{\operatorname{Cl}}}
\def\Max{{\operatorname{-Max}}}
\def\XX{{\bf{X}}}
\def\YY{{\bf{Y}}}
\def\BBB{{\mathcal B}}
\def\inv{{\operatorname{inv}}}
\def\emph{\it}
\def\Int{{\operatorname{Int}}}
\def\Spec{\operatorname{Spec}}
\def\Bin{{\operatorname{B}}}
\def\n{\operatorname{b}}
\def\N{{\operatorname{GB}}}
\def\BC{{\operatorname{BC}}}
\def\dlog{\frac{d \log}{dT}}
\def\Sym{\operatorname{Sym}}
\def\Nr{\operatorname{Nr}}
\def\lbrack{{\{}}
\def\rbrack{{\}}}
\def\burnside{\operatorname{B}}
\def\Sym{\operatorname{Sym}}
\def\Hom{\operatorname{Hom}}
\def\Inj{\operatorname{Inj}}
\def\Aut{{\operatorname{Aut}}}
\def\Mor{{\operatorname{Mor}}}
\def\Map{{\operatorname{Map}}}
\def\CMap{{\operatorname{CMap}}}
\def\GMaps{G{\operatorname{-Maps}}}
\def\Fix{{\operatorname{Fix}}}
\def\res{{\operatorname{res}}}
\def\ind{{\operatorname{ind}}}
\def\inc{{\operatorname{inc}}}
\def\coind{{\operatorname{cnd}}}
\def\Equiv{{\mathcal{E}}}
\def\W{\operatorname{W}}
\def\F{\operatorname{F}}
\def\witt{\operatorname{gh}}
\def\ngh{\operatorname{ngh}}
\def\Fm{{\operatorname{Fm}}}
\def\bij{{\iota}}
\def\mk{{\operatorname{mk}}}
\def\km{{\operatorname{mk}}}
\def\VV{{\bf{V}}}
\def\ff{{\bf{f}}}
\def\ZZ{{\mathbb Z}}
\def\Zhat{{\widehat{\mathbb Z}}}
\def\CC{{\mathbb C}}
\def\PP{{\mathbf p}}
\def\DD{{\mathbb D}}
\def\EE{{\mathbb E}}
\def\MM{{\mathbb M}}
\def\JJ{{\mathbb J}}
\def\NN{{\mathbb N}}
\def\RR{{\mathbb R}}
\def\QQ{{\mathbb Q}}
\def\FF{{\mathbb F}}
\def\mm{{\mathfrak m}}
\def\nn{{\mathfrak n}}
\def\jj{{\mathfrak j}}
\def\aaa{{{{\mathfrak a}}}}
\def\bbb{{{{\mathfrak b}}}}
\def\ppp{{{{\mathfrak p}}}}
\def\qqq{{{{\mathfrak q}}}}
\def\PPP{{{{\mathfrak P}}}}
\def\BB{{\mathfrak B}}
\def\jj{{\mathfrak J}}
\def\LL{{\mathfrak L}}
\def\qq{{\mathfrak Q}}
\def\rr{{\mathfrak R}}
\def\cc{{\mathfrak S}}
\def\TT{{\mathcal{T}}}
\def\SS{{\mathcal S}}
\def\UU{{\mathcal U}}
\def\AA{{\mathcal A}}
\def\BB{{\mathcal B}}
\def\Primes{{\mathcal P}}
\def\genS{{\langle S \rangle}}
\def\genT{{\langle T \rangle}}
\def\bT{\mathsf{T}}
\def\bD{\mathsf{D}}
\def\bC{\mathsf{C}}
\def\VV{{\bf V}}
\def\ff{{\bf f}}
\def\uu{{\bf u}}
\def\aa{{\bf{a}}}
\def\bb{{\bf{b}}}
\def\zero{{\bf 0}}
\def\rad{\operatorname{rad}}
\def\End{\operatorname{End}}
\def\id{\operatorname{id}}
\def\mod{\operatorname{mod}}
\def\im{\operatorname{im}\,}
\def\ker{\operatorname{ker}}
\def\coker{\operatorname{coker}}
\def\ord{\operatorname{ord}}
\def\li{\operatorname{li}}
\def\Ei{\operatorname{Ei}}
\def\Ein{\operatorname{Ein}}
\def\Ri{\operatorname{Ri}}
\def\Rie{\operatorname{Rie}}
\def\degl{\operatorname{deglog}}

\newtheorem{theorem}{Theorem}[section]
\newtheorem{proposition}[theorem]{Proposition}
\newtheorem{corollary}[theorem]{Corollary}
\newtheorem{lemma}[theorem]{Lemma}

\theoremstyle{definition}
\newtheorem{definition}[theorem]{Definition}
\newtheorem{remark}[theorem]{Remark}
\newtheorem{conjecture}[theorem]{Conjecture}
\newtheorem{problem}[theorem]{Problem}
\newtheorem{speculation}[theorem]{Speculation}
\newtheorem{example}[theorem]{Example}

 \newenvironment{map}[1]
   {$$#1:\begin{array}{rcl}}
   {\end{array}$$
   \\[-0.5\baselineskip]
 }

 \newenvironment{map*}
   {\[\begin{array}{rcl}}
   {\end{array}\]
   \\[-0.5\baselineskip]
 }

 \newenvironment{nmap*}
   {\begin{eqnarray}\begin{array}{rcl}}
   {\end{array}\end{eqnarray}
   \\[-0.5\baselineskip]
 }

 \newenvironment{nmap}[1]
   {\begin{eqnarray}#1:\begin{array}{rcl}}
   {\end{array}\end{eqnarray}
   \\[-0.5\baselineskip]
 }

\newcommand{\eq}{eq.\@\xspace}
\newcommand{\eqs}{eqs.\@\xspace}
\newcommand{\diagram}{diag.\@\xspace}

\numberwithin{equation}{section}

%\begin{frontmatter}

\title{Harmonic numbers and the prime counting function}

\author{Jesse Elliott}

\maketitle

\begin{abstract}
We provide approximations to the prime counting function by various discretized versions of the logarithmic integral function, expressed solely in terms of the harmonic numbers.  We demonstrate with explicit error bounds that these approximations are at least as good as the logarithmic integral approximation.  As a corollary, we provide some reformulations of the Riemann hypothesis in terms of the prime counting function and the harmonic numbers. \\


\end{abstract}

\bigskip

\tableofcontents

\section{Introduction}

This paper concerns the function $\pi: \RR_{>0} \longrightarrow \RR$ that for any $x > 0$ counts the number of primes less than or equal to $x$: $$\pi(x) =   \# \{p \leq x: p \mbox{ is prime}\}, \quad x > 0.$$  The function $\pi(x)$ is known as the {\it prime counting function}. We call the related function $\PP: \RR_{> 0} \longrightarrow \RR$ defined by $$\PP(x) = \frac{\pi(x)}{x}, \quad x > 0,$$ the {\it prime density function}.   The celebrated {\it prime number theorem}, proved independently by de la Vall\'ee Poussin \cite{val1} and Hadamard \cite{had}  in 1896,  states that
\begin{align*}
\pi(x) \sim \frac{x}{\log x} \ (x \to \infty),
\end{align*}
where $\log x$ is the natural logarithm.  
It is known, however, that the {\it logarithmic integral function} $$\li(x) =  \int_0^x \frac{dt}{\log t}, \quad x > 0,$$
(where the integral assumes the Cauchy principal value and $\li(1) = -\infty$), provides a better approximation to $\pi(x)$ than any algebraic function of $\log x$.  The {\it prime number theorem with error term}, proved by de la Vall\'ee Poussin in 1899 \cite{val2}, states that
$$\pi(x)-\li(x) = O \left(x e^{-C\sqrt{\log x}} \right) \ (x \to \infty)$$
for some constant $C  > 0$.  De la Vall\'ee Poussin's result has since been improved to
$$\pi(x)  - \li(x) = O\left(xe^{ - A(\log x)^{3/5}(\log \log x)^{-1/5}}\right) \ (x \to \infty),$$ 
where  $A  = 0.2098$ \cite{ford}, which is the strongest known $O$ bound on $\pi(x)-\li(x)$ to date.

Proofs of the strongest known bounds on the error $\pi(x) - \li(x)$  are based on Riemann's explicit formula for $\pi(x)$ in terms of the zeros of the Riemann zeta function $\zeta(s)$ and advanced methods for verifying zero-free regions of $\zeta(s)$ in the critical strip $0 \leq \operatorname{Re} s \leq 1$.   The celebrated {\it Riemann hypothesis} states that all such zeros lie on the line $\operatorname{Re} s = \frac{1}{2}$.  As is now well known, von Koch proved  in 1901 \cite{koch} that the  Riemann hypothesis is equivalent to 
$$\pi(x)- \li(x) = O(\sqrt{x} \log x) \ (x \to \infty).$$
   It is known, more generally, that if $$\delta = \sup\{\text{Re}(s): s \in \CC, \, \zeta(s) = 0\}$$ denotes the supremum of the real parts of the zeros of $\zeta(s)$, then $\frac{1}{2} \leq \delta \leq 1$, and $\delta$ is the least $\alpha \in \RR$ such that 
\begin{align}\label{ess}
\pi(x)- \li(x) = O(x^\alpha \log x) \ (x \to \infty).
\end{align}
(See, for example, \cite[Theorem 15.2 and Section 13.1.1 Exercise 1]{mont}.)
Moreover, the Riemann hypothesis is equivalent to $\delta = \frac{1}{2}$.

For every positive integer $n$, let $H_n = \sum_{k = 1}^n \frac{1}{k}$ denote the $n$th {\it harmonic number}.  The {\it summatory function} of a function $f(x)$ is the function $\sum_{k = 1}^n f(k)$.   Thus, the function $H_n$ is the summatory  function of  $f(x) = \frac{1}{x}$.   Summatory functions are {\it discrete integrals} in the sense that $\sum_{k = N}^n f(k) = \int_N^n f(x) d \nu(x)$ for all integers $n \geq N$, where $\nu$ is the unique discrete measure with respect to Lebesgue measure that is supported on $\ZZ$ with all weights equal to $1$.  Thus, the $n$th harmonic number $H_n = \int_1^n \frac{1}{x}d\nu(x)$ is a discrete integral of $\frac{1}{x}$ and is in this sense a ``discrete natural logarithm.''   Not unexpectedly, one has  $$H_n \sim \int_{1}^n \frac{dx}{x} = \log n \ (n \to \infty),$$ and, more precisely, the limit
$$\lim_{n \to \infty}\left( H_n -\log n  \right) = \gamma = 0.577215664901\ldots,$$
 known as the {\it Euler--Mascheroni constant}, 
is finite, and represents a precise measure of the  discrepancy between the natural logarithm and the ``discrete natural logarithm.''

Because $H_n \sim \log n \ (n \to \infty)$, the prime number theorem is equivalent to
$$\pi(n) \sim \frac{n}{H_n} \ (n \to \infty),$$ where
$\frac{n}{H_n}$ is also the harmonic mean of the integers $1,2,3,\ldots,n$.  This simple observation, alongside an inequality equivalent to the Riemann hypothesis involving the sum of divisors function and the harmonic numbers discovered by J.\ Lagarias, described below, led us to wonder if the harmonic numbers could be used to provide approximations to $\pi(n)$ that are better than $\frac{n}{H_n}$---or, ideally, even as good as $\li(n)$.  The former problem in part inspired the paper \cite{ell0}, where we provide various asymptotic expansions of the prime counting function, including several involving the harmonic numbers, such as the (divergent) asymptotic continued fraction expansion
$$\PP(e^\gamma n) \sim \, \cfrac{\sfrac{1}{H_n}}{1 - \cfrac{\sfrac{1}{H_n}}{1-\cfrac{\sfrac{1}{H_n}}{1-\cfrac{\sfrac{2}{H_n}}{1-\cfrac{\sfrac{2}{H_n}}{1 - \cfrac{\sfrac{3}{H_n}}{1-\cfrac{\sfrac{3}{H_n}}{1-\cdots}}}}}}} \ \ (n \to \infty)$$
of $\PP(e^\gamma n)$ \cite[Corollary 4.5]{ell0}, where 
$$e^\gamma = \lim_{n \to \infty} \frac{e^{H_n}}{n} =  1.781072417990\ldots,$$
and also where $$e^\gamma = \lim_{x \to \infty} \frac{\PP(x)}{\prod_{p \leq x} \left( 1-\frac{1}{p}\right)}$$
due to the prime number theorem and the third of Mertens' famous three theorems of 1874 \cite{mert}.  In 1984, G.\ Robin proved \cite{rob} that the Riemann hypothesis holds if and only if 
$$\sum_{d | n} d \leq e^\gamma n \log \log n, \quad \forall n \geq 5041.$$
Since by Mertens' second theorem and a 1913 result of Gronwall \cite{gron} one also has
$$e^\gamma = \limsup_{n \to \infty} \frac{\sum_{d | n} d }{  n \log \log n}  =  \limsup_{n \to \infty} \frac{\sum_{d | n} \frac{1}{d} }{   \sum_{p \leq n} \frac{1}{p}},$$
the constant $e^\gamma$ in Robin's equivalence is the best possible.  In 2000, J.\ Lagarias used Robin's result to show  \cite{lag} that the Riemann hypothesis holds if and only if
$$\sum_{d | n} d <  H_n + e^{H_n}\log H_n, \quad \forall n > 1,$$
if and only if
$$\sum_{d | n} d <  e^{H_n}\log H_n, \quad \forall n > 60.$$
Lagarias' inequalities are closely related to Robin's because, by asymptotics noted earlier, one has
$$e^\gamma n \log \log n  \sim e^{H_n} \log H_n \sim H_n + e^{H_n} \log H_n \ (n \to \infty).$$

The three ``elementary'' reformulations of the Riemann hypothesis noted above concern the sum of divisors function rather than the prime counting function.  In this paper, we provide several reformulations of the Riemann hypothesis that are expressed solely in terms of the harmonic numbers and the prime counting function.  For example, we show in Section 5 that the Riemann hypothesis holds if and only if
$$\PP(e^\gamma n) =  \frac{1}{n} \sum_{k = 1}^{n} \frac{1}{H_k} + O \left(\frac{H_n}{ \sqrt{n}}\right) \ (n \to \infty),$$
if and only if
 $$\left|\PP(e^\gamma n) -\frac{1}{n} \sum_{k = 2}^{n-1}   \frac{1}{H_k}   \right| < \frac{1}{8 \pi e^{\gamma/2}} \frac{H_n}{\sqrt{n}}+\frac{1+0.4986013304}{n}, \quad \forall n \geq 1.$$
Moreover, any choice of larger constants still yields a Riemann hypothesis equivalent, so, for example, since $\frac{1}{8 \pi e^{\gamma/2}} = \frac{1}{33.541358\ldots} < \frac{1}{33},$
the Riemann hypothesis is also equivalent to
 $$\left|\PP(e^\gamma n) -\frac{1}{n} \sum_{k = 2}^{n-1}   \frac{1}{H_k}   \right| <  \frac{1}{33} \frac{H_n}{\sqrt{n}}+\frac{3}{2n}, \quad \forall n \geq 1.$$
Such a reformulation of the Riemann hypothesis is noteworthy because it makes no mention of transcendental functions and the only numbers in the given inequality that may not be rational are  $e^\gamma$ and $\sqrt{n}$.

The second and third of our Riemann hypothesis equivalents above are made possible by a well-known reformulation due to L.\ Schoenfeld \cite{sch}: the Riemann hypothesis holds if and only if
\begin{align}\label{sch}
|\pi (x)-\operatorname {li} (x)|<{\frac {1}{8\pi }}{\sqrt {x}}\log x, \quad \forall x \geq 2657.
\end{align}
The constant $0.4986013304$ in the second Riemann hypothesis equivalent can be replaced with any other upper bound of the limit
$$\kappa =  \lim_{n \to \infty} \left (\frac{\li(e^{\gamma}n) }{e^\gamma}-\sum_{k = 1}^{n-1} \frac{1}{H_k}\right) \approx 0.4986.$$
The limit $\kappa$ exists because the sequence $\frac{\li(e^{\gamma}n) } {e^\gamma} -\sum_{k = 1}^{n-1} \frac{1}{H_k}$ is positive, strictly increasing, and bounded above, and therefore bounded above by $\kappa$.  Our results allow us to compute tight upper and lower bounds of constants like $\kappa$, so that, for example, one has
$$ 0.4985987518 < \kappa < 0.4986013304.$$
Moreover, since $\li(x) > \pi(x) + \frac{1.49}{2}$ for all $x \geq 6$ for which the value of $\pi(x)$ is known, the sum $ \sum_{k = 2}^{n-1} \frac{1}{H_k}$ is closer to $\frac{\pi(e^\gamma n)}{e^\gamma}$ than is $\frac{\li(e^\gamma n)}{e^\gamma}$ for all $n \geq 6$ for which the value of $\pi(e^\gamma n)$ is known.  Thus, approximations of the prime counting function using harmonic numbers can indeed be worthy rivals of the standard logarithmic integral approximation.

In Section 3, we prove that, for all $t \in \RR$, one has
\begin{align*}
\frac{\li (e^t n)}{e^t}  & =  \sum_{\mu e^{-t}  \leq k < n} \frac{1}{H_k-\gamma+t} + \beta_n(t)  \\  & = \sum_{\mu e^{-t}  \leq k < n}   \frac{1}{H_k-\gamma+t} + \beta(t)- \frac{1+o(1)}{12n(\log n)^2} \ (n \to \infty)
\end{align*}
for unique error functions $\beta_n(t) > 0$ and $\beta(t) > 0$ with $\beta(t) = \lim_{n \to \infty} \beta_n(t)$, where $\mu = 1.451369234883\ldots$ is the unique positive zero of $\li(x)$, called the {\emph Ramanujan--Soldner constant}. 
We also prove explicit bounds on $\beta_n(t)$ and $\beta(t)$ in terms of $t$ that allow us to compute $\beta(t)$ to any desired degree of accuracy,  and we show, for example, that
$$\beta(t) <  \limsup_{t \to -\infty} \beta(t) =  \frac{1}{\log \mu}= 2.684510350820\ldots,$$ 
$$\beta(t)  =  \frac{\li(e^t \lceil \mu e^{-t} \rceil)}{e^t} +\frac{(1 +o(1))e^t}{12\mu (\log \mu)^2} \ (t \to -\infty),$$
and $$\beta(t) = \frac{1}{t} + O \left( \frac{1}{t^2}\right) \ (t \to \infty),$$
where $\lceil x \rceil$ (resp., $\lfloor x \rfloor$) denotes the ceiling (resp., floor) of $x$ for any real number $x$.  Note that the constant $\kappa \approx 0.4986$ introduced earlier is precisely $\beta(\gamma)$. 

 In Section 4, we use the results noted above to make precise the approximation
 $$\frac{\pi (e^t n)}{e^t} \approx \sum_{\mu e^{-t}  \leq k < n}   \frac{1}{H_k-\gamma+t}, \quad n \geq  \mu e^{-t},$$
from which we derive our Riemann hypothesis equivalents in Section 5.  Analogous to the integral representation
$\int_{\mu e^{-t}}^x \frac{du}{t+\log u}$
of $\frac{\li (e^t x)}{e^t}$,  the sum $\sum_{\mu e^{-t} \leq k < x}   \frac{1}{H_k-\gamma+t}$ can be represented as the discrete integral
$\int_{\mu e^{-t}}^{x+0^-} \frac{d\nu(u)}{H_u-\gamma+t}$.  Our approximation to $\frac{\pi (e^t n)}{e^t}$ above  is therefore a doubly discretized version of the logarithmic integral.

Since the Riemann hypothesis, for all we currently know,  could be false, we find it useful to generalize our  reformulations of the  hypothesis to  unconditional results expressed in terms of the supremum of the real parts of the zeros of the Riemann zeta function.  Eq.\ (\ref{ess}) is the quintessential example of such a generalization.  In fact, all of the ``heavy lifting'' regarding the prime counting function in this paper is accomplished by Eqs.\ (\ref{ess}) and (\ref{sch}) and a theorem of Montgomery and Vaughan (Theorem \ref{mv}).  We are thus able to focus most of our attention on using the harmonic numbers to approximate the logarithmic integral using elementary analysis.

I would like to thank Sean Lubner and Daniel Brice for writing Python code to check the inequalities in Corollaries \ref{RH5} through \ref{RH9} for small values of $n$.

\section{Approximating $\log x$ with harmonic numbers}

In this section, we list some properties of the harmonic numbers that  form the basis for our results. 
 
From the functional equation
$$\Gamma(z+1) = z \Gamma(z), \quad z \in \CC\backslash \{0,-1,-2,-3,\ldots\}$$
for the gamma function $\Gamma(z)$ follows, by logarithmic differentiation, the functional equation
$$\Psi(z+1) = \frac{1}{z}+\Psi(z), \quad z \in \CC\backslash \{0,-1,-2,-3,\ldots\}$$
for the {\it digamma function} $\Psi(z) = \frac{\Gamma'(z)}{\Gamma(z)}$.  Since $\Psi(1) = \Gamma'(1) = -\gamma$ and $H_0 = 0$, it follows that the harmonic numbers $H_n$ are interpolated by the complex function
\begin{align}\label{Hpsi}
H_z = \Psi(z+1)+ \gamma = \sum_{k = 1}^\infty \left(\frac{1}{k}-\frac{1}{z+k}\right) = \lim_{n \to \infty} \left( H_n - \sum_{k = 1}^n \frac{1}{z+k}\right), \nonumber \\   z \in \CC\backslash \{-1,-2,-3,\ldots\}.
\end{align}
It is known that 
$$H_z -\gamma -\log z \sim \frac{1}{2z} \ (z \to \infty),$$
and, more generally, by the Euler-Maclaurin formula, that one has the (divergent) asymptotic expansion
$$H_z -\gamma -\log z - \frac{1}{2z} \sim  \sum_{k = 1}^\infty \frac{-B_{2k}}{2kz^{2k}} \ (z \to \infty),$$
where $B_k$ is the $k$th Bernoulli number.  From this well-known expansion  follows the asymptotic expansion
$$H_{z-\sfrac{1}{2}} -\gamma -\log z \sim  \sum_{k = 1}^\infty \frac{(1-2^{-2k+1})B_{2k} }{2kz^{2k}} \ (z \to \infty).$$
 From the latter expansion and \cite[Theorem 8]{alz}, one can show that
$$\sum_{k = 1}^{n+1} \frac{(1-2^{-2k+1})B_{2k} }{2kx^{2k}} < H_{x-\sfrac{1}{2}} -\gamma -\log x <  \sum_{k = 1}^{n} \frac{(1-2^{-2k+1})B_{2k} }{2kx^{2k}}$$
for all $x > 0$ and all odd positive integers $n$.
Thus, for example, one has
\begin{align*} \frac{1}{24(x+\sfrac{1}{2})^2} - \frac{7}{960(x+\sfrac{1}{2})^4}< H_x - \gamma - \log (x + \sfrac{1}{2}) < \frac{1}{24(x+\sfrac{1}{2})^2}, \quad \forall x > -\sfrac{1}{2}.
\end{align*}
and therefore
\begin{align}\label{harmony} \frac{1}{24(x+1)^2} < H_x - \gamma - \log (x + \sfrac{1}{2}) < \frac{1}{24(x+\sfrac{1}{2})^2}, \quad \forall x > -\sfrac{1}{2}.
\end{align}
In particular, $\log(x +\sfrac{1}{2})+\gamma$ is an excellent approximation for $H_x$, and, correspondingly,  $H_{x-\sfrac{1}{2}} - \gamma$  is an excellent approximation for $\log x$.  Since
$$H_x - \gamma - \log x \sim \frac{1}{2x} \ (x \to \infty),$$
while $$H_x - \gamma - \log(x +\sfrac{1}{2}) \sim \frac{1}{24 x^2} \ (x \to \infty),$$ the advantage  gained by shifting the log by $\sfrac{1}{2}$ is clear.  This makes sense heuristically because of the advantage, for monotonic functions, of the midpoint rule over left-hand or right-hand Riemann sums.  See  Figures \ref{har5} and \ref{har4} for a graphical comparison of the approximations above.

\begin{figure}[ht!]
\centering
\includegraphics[width=100mm]{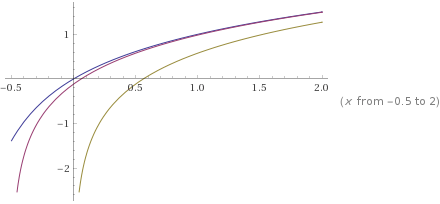}
\caption{Graphs of $H_x > \log (x + \sfrac{1}{2}) + \gamma > \log x + \gamma$  on $(-\sfrac{1}{2},2]$  \label{har5}}
\end{figure}

\begin{figure}[ht!]
\centering
\includegraphics[width=100mm]{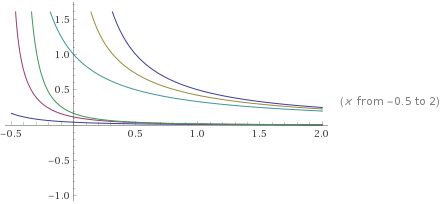}
\caption{Graphs of $\frac{1}{24(x+1)^2}$, $H_x -\gamma- \log (x + \sfrac{1}{2})$, $\frac{1}{24(x+\sfrac{1}{2})^2}$, $\frac{1}{2(x+\sfrac{1}{2})}$,  $H_x -\gamma-\log x$, $\frac{1}{2x}$ on $(-\sfrac{1}{2},2]$, ordered from smallest to largest on $[0,\infty)$  \label{har4}}
\end{figure}

By the functional equation $$H_z = \frac{1}{z}+ H_{z-1} , \quad  z \in \CC\backslash \{0,-1,-2,-3,\ldots\}$$ for $H_z$, Eq.\ (\ref{harmony}) generalizes as follows.

\begin{proposition}\label{harm}
For  all $x  \in \RR\backslash \{-1,-2,-3,\ldots\}$ and all integers  $n> x-\sfrac{1}{2}$, one has
\begin{align*}
\frac{1}{24(x+n+1)^2}  \leq ( H_x -\gamma)  -\left(\log (x+n+\sfrac{1}{2})-\sum_{k = 1}^{n} \frac{1}{x+k} \right) \leq \frac{1}{24(x+n+\sfrac{1}{2})^2}.
\end{align*}
Consequently, for all $x  \in \RR\backslash \{-1,-2,-3,\ldots\}$ one has
$$\Psi(x+1) =  H_x -\gamma = \log (x+n+\sfrac{1}{2}) -\sum_{k = 1}^{n} \frac{1}{x+k} + \frac{1+ o(1)}{24n^2}  \ (n \to \infty).$$
\end{proposition}

\section{Approximating $\li(x)$ with harmonic numbers}

Let $\mu = 1.451369234883\ldots$ denote the {\it Ramanujan--Soldner constant},  which by definition is the unique positive zero of $\li(x)$, or equivalently the unique positive real number $\mu$ such that $\li(x) = \int_\mu^x \frac{dx}{\log x}$ for all $x > 1.$
In this section we  make precise the approximation
 $$\frac{\li (e^t n)}{e^t} \approx \sum_{k = \lceil \mu e^{-t} \rceil}^{n-1}   \frac{1}{H_k-\gamma+t}, \quad \forall n \geq \lceil \mu e^{-t} \rceil.$$  More specifically, we prove the following.

\begin{theorem}\label{maintheorem}
For all $t \in \RR$, one has
\begin{align*}
\frac{\li (e^t n)}{e^t}  & =  \sum_{k = \lceil \mu e^{-t} \rceil}^{n-1}    \frac{1}{H_k-\gamma+t} + \beta_n(t)  \\  & = \sum_{k =  \lceil \mu e^{-t} \rceil}^{n-1}    \frac{1}{H_k-\gamma+t} + \beta(t)- \frac{1+o(1)}{12n(\log n)^2} \ (n \to \infty)
\end{align*}
for unique error functions $\beta_n(t) > 0$ and $\beta(t) > 0$ with $\beta(t) = \lim_{n \to \infty} \beta_n(t)$.
\end{theorem}

To prove the theorem we require the following notation.  %Existence of all of the limits involved is proved within.

\begin{definition} Let $N$ be a positive integer, and let $t \in \RR$ with $t > -\log N$.
\begin{enumerate}
\item Let
$$\theta_n(t,N) =  \int_{N}^n \frac{dx}{t + \log x} -\sum_{k = N}^{n-1}  \frac{1}{H_k-\gamma+t}$$
for all $n \geq N$,  and  let
$$\theta(t,N) =  \lim_{n \to \infty} \theta_n(t,N).$$
\item Let
$$\eta_n(t,N) =   \sum_{k = N}^{n} \left(\frac{1}{t+\log(k+\sfrac{1}{2})}-\frac{1}{H_n-\gamma+t}\right).$$
for all $n \geq N$,  and  let
$$\eta(t,N) =  \lim_{n \to \infty} \eta_n(t,N).$$
\item Let
 $$\delta_n(t,N) =  \int_{N}^n \frac{dx}{t+\log x} - \sum_{k = N}^{n-1}  \frac{1}{\log (k+\sfrac{1}{2})+t}$$
for all $n \geq N$, and let
 $$\delta(t,N) =  \lim_{n \to \infty} \delta_n(t,N).$$
\end{enumerate}
\end{definition}

Our first goal is to prove the following theorem, from which we will then deduce Theorem \ref{maintheorem}.

\begin{theorem}\label{mainprop}
Let $N$ be a positive integer, and let $t > -\log N$.  
\begin{enumerate}
\item  The sequence $$\theta_n(t,N)= \int_N^{n} \frac{dx}{t+\log x} - \sum_{k = N}^{n-1} \frac{1}{H_k-\gamma+t}, \quad n > N$$
is positive, strictly increasing, and bounded above.    Consequently, the limit
\begin{align*}
\theta(t,N) = \lim_{n \to \infty} \theta_n(t,N) > 0
\end{align*}
exists.
\item For all $n \geq N$, one has
\begin{align*}
&\int_{n+1}^{\infty} \frac{dx}{24x^2 (t+\log x)^2} +
\frac{1}{24(n+1)(t+\log (n+1))^2} \\
 & \qquad \leq \theta(t,N) - \theta_n(t,N)  = \theta(t,n) \\ & \qquad \leq \int_n^\infty \frac{dx}{24x^2 (t+\log x)^2}  +\frac{1}{24(n+\sfrac{1}{2})(t+\log (n+\sfrac{1}{2}))^2} \\ & \qquad \qquad  + \frac{1}{24n^2(t+\log n)^2}  +  \frac{1}{12n^2(t+\log n)^3}.
\end{align*}
\item One has
$$\theta(t,N) = \theta_{n}(t,N) + \frac{1+o(1)}{12n(\log n)^2} \ (n \to \infty)$$
and
$$\theta(t,N) = O\left( \frac{1}{N t^2}\right) \ (t \to \infty),$$
where the $O$ constant does not depend on $N$.
\end{enumerate}
\end{theorem}

We divide the proof of Theorem \ref{mainprop} into two main steps, based on the  equality $$\theta_n(t,N) = \eta_{n-1}(t,N) + \delta_{n-1}(t,N).$$   First, we prove the following analogue of the theorem for the sequence $\eta_n(t,N)$.

\begin{proposition}\label{mainprop1}
Let $N$ be a positive integer, and let $t > -\log N$.  
\begin{enumerate}
\item The sequence $$\eta_n(t,N) = \sum_{k = N}^n \frac{1}{t+\log(k+\sfrac{1}{2})}-\sum_{k = N}^n \frac{1}{H_k-\gamma+t}, \quad n \geq N$$ 
 is positive, strictly increasing, and bounded above.  Consequently, the limit
$$ \eta(t,N) =  \lim_{n \to \infty} \eta_n(t,N) > 0$$
exists. 
\item For all $n \geq N$, one has
\begin{align*}
\frac{1}{24}\int_{n+1}^\infty \frac{dx}{x^2 (t+\log x)^2}\leq  \eta(t,N)-\eta_{n-1}(t,N) = \eta(t,n)  \leq  \frac{1}{24}\int_n^\infty \frac{dx}{x^2 (t+\log x)^2}.
\end{align*}
\item  For all $n \geq N$, one has
\begin{align*}
&\frac{1}{24(n+1)(t+\log (n+1))^2} -\frac{1}{12(n+1)(t+\log(n+1))^3} \\
& \qquad < \eta(t,N) -\eta_{n-1}(t,N) <  \frac{1}{24n(t+\log n)^2}.
\end{align*}
Consequently, one also has
$$\eta(t,N) = \eta_{n-1}(t,N) + \frac{1+o(1)}{24n(\log n)^2} \ (n \to \infty)$$
and
$$\eta(t,N) = O\left( \frac{1}{N t^2}\right) \ (t \to \infty),$$
where the $O$ constant does not depend on $N$.
\end{enumerate}
\end{proposition}

\begin{proof}
Let $n \geq N$.  By Eq.\ (\ref{harmony}), for all $k > e^{-t}$, hence for all $k \geq N$, one has
\begin{align*}
0< \frac{1}{t+\log(k+\sfrac{1}{2})}- \frac{1}{H_k-\gamma+t}  & = \frac{ H_k - \gamma - \log (k + \sfrac{1}{2})}{(t+\log(k+\sfrac{1}{2}))(H_k-\gamma+t)}  \\
  & <  \frac{1}{24(k+\sfrac{1}{2})^2(t+\log(k+\sfrac{1}{2}))^2}.
\end{align*}
Therefore, the sequence $\eta_n(t,N)$ is positive and strictly increasing, and one has
\begin{align*}
 \eta_n(t,N)  <  \frac{1}{24}\sum_{k = N}^n \frac{1}{(k+\sfrac{1}{2})^2(t+\log(k+\sfrac{1}{2}))^2}  <  \frac{1}{24}\int_N^{n+1}  \frac{dx}{x^2(t+\log x)^2}.
\end{align*}
Similarly, one has
\begin{align*}
\eta_n(t,N)   >\frac{1}{24} \sum_{k = N}^{n} \frac{1}{(k+1)^2(t+\log (k+1))^2} > \int_{N+1}^{n+2}  \frac{dx}{24x^2(t+\log x)^2},
\end{align*}
and therefore
$$0 < \frac{1}{24}\int_{N+1}^{n+2} \frac{dx}{x^2 (t+\log x)^2} < \eta_n(t,N) < \frac{1}{24}\int_N^{n+1} \frac{dx}{x^2 (t+\log x)^2}.$$   Statement (1) follows.  Taking the limit as $n \to \infty$, we see that
$$\frac{1}{24}\int_{N+1}^\infty \frac{dx}{x^2 (t+\log x)^2} \leq  \eta(t,N) \leq \frac{1}{24}\int_N^\infty \frac{dx}{x^2 (t+\log x)^2}.$$
Statement (2) then follows from the inequality above  and the fact that
$$\eta(t, N)-\eta_{n-1}(t,N) = \eta(t,n)$$
for all $n \geq N$.  Finally, statement (3) follows from statement (2) and Lemma \ref{lambda} below.
\end{proof}

\begin{lemma}\label{lambda}
Let $t, r \in \RR$ with $r > 0$ and $t > -\log r$.   One has
$$\int_r^\infty \frac{dx}{x^2(t+\log x)^2} = e^t \li\left(\frac{1}{re^{t}}\right) + \frac{1}{r (t+\log r)} > 0.$$
Moreover, for every even positive integer $n$, one has
$$\sum_{k = 2}^{n+1} \frac{(-1)^k (k-1)!}{r(t+\log r)^k} <\int_r^\infty \frac{dx}{x^2(t+\log x)^2}<  \sum_{k = 2}^{n} \frac{(-1)^k (k-1)!}{r(t+\log r)^k}.$$
In particular, one has
$$\frac{1}{r(t+\log r)^2} -\frac{2}{r(t+\log r)^3}< \int_r^\infty \frac{dx}{x^2(t+\log x)^2}<   \frac{1}{r(t+\log r)^2}.$$
\end{lemma}

\begin{proof}
 The exact expression for the integral is easily verified by integration by parts.  Since it is known that
$$-\sum_{k = 0}^{n} \frac{(-1)^k k! x}{(\log x)^{k+1}} < \li(x)  <   -\sum_{k = 0}^{n-1} \frac{(-1)^k k!x}{(\log x)^{k+1}}$$
for all $0 < x < 1$ and all even positive integers $n$, letting $x = \frac{1}{re^{t}}$, we see that
$$-\sum_{k = 1}^{n} \frac{(-1)^k k!}{r(t+\log r)^{k+1}} <e^t \li\left(\frac{1}{re^{t}}\right) + \frac{1}{r (t+\log r)} <   -\sum_{k = 1}^{n-1} \frac{(-1)^k k!}{r(t+\log r)^{k+1}}.$$
The lemma follows.
\end{proof}

The second and final step in the proof of Theorem \ref{mainprop} is to prove the following analogue of the theorem for the sequence $\delta_n(t,N)$.

\begin{proposition}\label{mainprop2}
Let $N$ be a positive integer, and let $t > -\log N$.  
\begin{enumerate}
\item The sequence $$ \delta_n(t,N) =   \int_N^{n+1} \frac{dx}{t+\log x} - \sum_{k = N}^n \frac{1}{t+\log(k+\sfrac{1}{2})}, \quad n \geq N$$
is positive, strictly increasing, and bounded above.  Consequently, the limit
$$ \delta(t,N) =  \lim_{n \to \infty} \delta_n(t,N) > 0$$
exists.
\item For all $n \geq N$, one has
\begin{align*}
& \frac{1}{24(n+1)(t+\log (n+1))^2} \\ &  \qquad \leq \delta(t,N)-\delta_{n-1}(t,N)  = \delta(t,n) \\
  & \qquad  \leq \frac{1}{24(n+\sfrac{1}{2})(t+\log (n+\sfrac{1}{2}))^2} +\frac{1}{24n^2(t+\log n)^2}+  \frac{1}{12n^2(t+\log n)^3}.
\end{align*}
\item One has
$$\delta(t,N) = \delta_{n-1}(t,N) + \frac{1+o(1)}{24n(\log n)^2} \ (n \to \infty)$$
and
$$\delta(t,N) = O\left( \frac{1}{N t^2}\right) \ (t \to \infty),$$
where the $O$ constant does not depend on $N$.
\end{enumerate}
\end{proposition}

%\begin{alignat*}{2}
% 1 & < a &&  +c \\
 %1 & <  a && +  b  \\ 
% 1 & < b && + c
%\end{alignat*}

\begin{proof}
Let $f(x) = \frac{1}{t+\log x}$, which is positive, decreasing, and concave up on $[N, \infty)$.  Then, also, $f'(x) = -\frac{1}{x(t+\log x)^2}$ is negative, increasing, and concave down on $[N,\infty)$, while $f''(x) = \frac{1}{x^2(t+\log x)^2} + \frac{2}{x^2(t+\log x)^3}$ is positive, decreasing, and concave up on $[N,\infty)$. In particular, by the well-known expression for the error from  the midpoint rule, one has
$$0< \frac{1}{24}f''(k+1) < \int_k^{k+1}  \frac{dx}{t+\log x} - \frac{1}{t+\log (k+\sfrac{1}{2})} < \frac{1}{24}f''(k)$$
and therefore
$$ \frac{1}{24} \sum_{k = N}^n  f''(k+1) <  \delta_n(t,N)  < \frac{1}{24} \sum_{k = N}^n  f''(k).$$
Moreover, one has
\begin{align*}
\sum_{k = N}^n f''(k) <  f''(N)+ \int_{N+\sfrac{1}{2}}^{n+\sfrac{1}{2}} f''(x) \, dx  =  f''(N) + f'(n+\sfrac{1}{2})-f'(N+\sfrac{1}{2}),
\end{align*}
while
\begin{align*}
\sum_{k = N}^n f''(k+1)  >  \int_{N+1}^{n+2} f''(x) \, dx   =  f'(n+2)-f'(N+1).
\end{align*}
It follows that, for all $n \geq N$, one has
\begin{align*}
&  \left. \left(\frac{1}{24x(t+\log x)^2} \right)\right|^{N+1}_{n+2}  \\   &  \qquad <\delta_n(t,N) \\  & \qquad <  \left. \left(\frac{1}{24x(t+\log x)^2} \right)\right|^{N+\sfrac{1}{2}}_{n+\sfrac{1}{2}}   + \frac{1}{24N^2(t+\log N)^2}  +  \frac{1}{12N^2(t+\log N)^3}.
\end{align*}
\item 
Statement (1) follows.  Statement (2) follows by taking a limit of the inequalities above as $n \to \infty$, and, finally, statement (3) follows immediately from (2).
\end{proof}

Theorem \ref{mainprop}, now, follows immediately from  Propositions \ref{mainprop1} and \ref{mainprop2}.  As a corollary, we obtain the following.

\begin{corollary}\label{monocor}
Let $N$ be a positive integer, and let $t > -\log N$.  One has
$$\theta(t,N) = \eta(t,N) + \delta(t,N).$$  Moreover, for all $n \geq N$, one has
\begin{align*}
0 \leq \theta_{n}(t,N) < \theta(t,N)  <  \theta_{n}(t,N) & + \frac{1}{24n(t+\log n)^2} +\frac{1}{24(n+\sfrac{1}{2})(t+\log (n+\sfrac{1}{2}))^2} \\
  & + \frac{1}{24n^2(t+\log n)^2}  +  \frac{1}{12n^2(t+\log n)^3}.
\end{align*}
\end{corollary}

A real function $f(x)$ on an interval $I \subseteq \RR$ is said to be {\emph strictly totally monotone on $I$} if $f(x)$ is continuous on $I$, infinitely differentiable on the interior of $I$, and satisfies $(-1)^k \frac{d^k}{dx^k} f(x) > 0$ for all $x$ in the interior of $I$ and for all nonnegative integers $k$. 

\begin{proposition}\label{monotone}
Let $N$ be a positive integer.   The function $\theta(t,N)$ strictly totally monotone on $(-\log N, \infty)$ with
\begin{align*}
(-1)^k \frac{d^k}{dt^k} \theta(t,N) & = (-1)^k \lim_{n \to \infty} \frac{d^k}{dt^k}\theta_n(t,N) \\ &=  k!\lim_{n \to \infty} \left(\int_N^n \frac{dx}{(t+\log x)^{k+1}} - \sum_{k = N}^{n-1}\frac{1}{(H_k-\gamma+t)^{k+1}} \right) > 0
\end{align*}
for all nonnegative integers $k$ and all $t \in (-\log N, \infty)$.
\end{proposition}

\begin{proof}
Because the sums $\sum_{n = N}^\infty \frac{1}{n(t+\log n)^2}$, $\sum_{n = N}^\infty \frac{1}{n^2(t+\log n)^2}$,   and $\sum_{n = N}^\infty \frac{1}{n^2(t+\log n)^3}$ converge for all  $t \in (-\log N, \infty)$ and are bounded on compact subsets of $(-\log N, \infty)$, Corollary \ref{monocor} implies that the convergence of $\theta_n(t,N)$ to $\theta(t,N)$ is uniform on compact subsets  of $(-\log N, \infty)$.  Therefore, the function $\theta(t,N)$ is continuous on $(-\log N, \infty)$.
In fact, the following argument shows that $\theta(t,N)$ is differentiable, with negative derivative, on $(-\log N, \infty)$.
First, note that 
$$\frac{d}{dt}\theta_n(t,N) = - \int_N^n \frac{dx}{(t+\log x)^2} + \sum_{k = N}^{n-1}\frac{1}{(H_k-\gamma+t)^2},$$
and then a straightforward repetition of our argument for $\theta_n(t,N)$ shows that, just as one has $\theta_n(t,N) > 0$ because
$$ \int_N^n \frac{dx}{t+\log x} > \sum_{k = N}^{n-1}\frac{1}{\log (k + \sfrac{1}{2})+t} >  \sum_{k = N}^{n-1}\frac{1}{H_k-\gamma+t},$$ one has
$\frac{d}{dt}\theta_n(t,N) < 0$ because $$ \int_N^n \frac{dx}{(t+\log x)^2} > \sum_{k = N}^{n-1}\frac{1}{(\log (k + \sfrac{1}{2})+t)^2} >  \sum_{k = N}^{n-1}\frac{1}{(H_k-\gamma+t)^2}.$$
One can then bound the error terms as with $\theta(t,N)$ and show that $\frac{d}{dt}\theta_n(t,N)$ converges uniformly on compact subsets of $(-\log N, \infty)$ to a limit
$$f(t,N) = \lim_{n \to \infty} \frac{d}{dt}\theta_n(t,N)$$
that is negative for all $t$.  It follows that $\theta(t, N)$ is differentiable on $(-\log N, \infty)$ with derivative $f(t,N)$.   The same analysis applies more generally to the expression
\begin{align*}
(-1)^k \frac{d^k}{dt^k}\theta_n(t,N) = k!\left(\int_N^n \frac{dx}{(t+\log x)^{k+1}} - \sum_{k = N}^{n-1}\frac{1}{(H_k-\gamma+t)^{k+1}} \right) > 0,
\end{align*}
for all positive integers $k$ and $n \geq N$.
\end{proof}

%Similar statements hold for the functions $\eta(t,N)$ and $\delta(t,N)$.

\begin{definition}\label{maindef}
Let $t, r \in \RR$ with $r > e^{-t}$.  Let
$$\beta_x(t,r) = \frac{\li (e^t x)}{e^t}  -\sum_{r \leq   k < x}   \frac{1}{H_k-\gamma+t}$$
for all $x \geq r$, 
and let
$$\beta(t,r) =  \lim_{n \to \infty} \beta_n(t,r) =  \lim_{x \to \infty} \beta_x(t,r).$$
\end{definition}

Let $t, r \in \RR$ with $r > e^{-t}$.   Since
 $$\beta(t,r) = \beta(t,\lceil r \rceil) =  \frac{\li(e^t \lceil r \rceil)}{e^t}+ \theta(t,\lceil r \rceil)$$
and
 $$\beta(t,r)-\beta_n(t,r) =\beta(t,\lceil r \rceil)- \beta_n(t,\lceil r \rceil) =  \theta(t,\lceil r \rceil) -  \theta_n(t,\lceil r \rceil) = \theta(t,n)$$
for all $n \geq \lceil r \rceil$, by Theorem \ref{mainprop} one has the following.

\begin{corollary}\label{maincor}
Let $t, r \in \RR$ with $r > e^{-t}$.  For all $n  \geq \lceil r \rceil$, one has
\begin{align*}
& \int_{n+1}^{\infty} \frac{dx}{24x^2 (t+\log x)^2} +
\frac{1}{24(n+1)(t+\log (n+1))^2} \\ & \qquad \leq \beta(t,r) - \beta_n(t,r)   \\ & \qquad \leq \int_n^\infty \frac{dx}{24x^2 (t+\log x)^2}  +\frac{1}{24(n+\sfrac{1}{2})(t+\log (n+\sfrac{1}{2}))^2} \\ & \qquad \qquad  + \frac{1}{24n^2(t+\log n)^2}  +  \frac{1}{12n^2(t+\log n)^3}.
\end{align*}
Moreover, one has
$$\beta(t,r) = \beta_{n}(t,r) + \frac{1+o(1)}{12n(\log n)^2} \ (n \to \infty).$$
\end{corollary}

Now, again, let $t,r \in \RR$ with $ r > e^{-t}$.  Then 
$$\frac{\li (e^t n)}{e^t}   =  \sum_{r\leq  k < n}   \frac{1}{H_k-\gamma+t}+ \beta_n(t,r), \quad \forall n \geq r,$$
so that, since by Corollary \ref{maincor} the quantity $\beta(t,r)-\beta_n(t,r)$ is small for $n \geq r$, one has
 $$\frac{\li (e^t n)}{e^t} \approx \sum_{r \leq k < n}   \frac{1}{H_k-\gamma+t} + \beta(t,r), \quad \forall n \geq r$$
in a sense made precise by the corollary.  It behooves us to choose $r = r_t$ in terms of $t$ so that the absolute value of the quantity $$\beta(t,r) = \beta(t,\lceil r \rceil) =  \frac{\li(e^t \lceil r \rceil)}{e^t}+ \theta(t,\lceil r \rceil)$$
is minimized.  Since $\theta(t,\lceil r \rceil)$ is always nonnegative,  by far the dominant and more unpredictable term in the expression above is $\frac{\li(e^t \lceil r \rceil)}{e^t} = \beta_{ \lceil r \rceil}(t,r)$.   At least the nonnegativity of $\beta(t,r)$ can be guaranteed as long as $ \frac{\li(e^t \lceil r \rceil)}{e^t}$ is nonnegative.  So it would be prudent to minimize that term subject to the constraint that $\frac{\li(e^t \lceil r \rceil)}{e^t}$ be nonnegative.

This can be achieved by employing the Ramanujan--Soldner constant $\mu$.  Clearly $\frac{\li(e^t \lceil r \rceil)}{e^t}$ is nonegative if and only if $e^t \lceil r \rceil \geq \mu$, and then the term is minimized for any $r$ with $\lceil r \rceil =  \lceil \mu e^{-t}\rceil $, e.g., for $r_t = \mu e^{-t},$ where $r_t > 0$ is uniquely determined by the equation $\li(r_te^t) = 0$.
Let $$R_t = \lceil r_t  \rceil  = \lceil \mu e^{-t}  \rceil.$$
Since $\li(x)$ is increasing for $x > 1$, one has
$$ 0  = \frac{\li(e^t r_t)}{e^t} \leq \frac{\li(e^t R_t)}{e^t}<  \frac{\li(e^t (r_t+1))}{e^t}  = \frac{\li(\mu+e^t)}{e^t}.$$
Moreover, the function $\frac{\li(\mu+e^t)}{e^t} $ is strictly decreasing with
$ \lim_{t \to \infty}  \frac{\li(\mu+e^t)}{e^t} =  0$
and
$$ \lim_{t \to -\infty}  \frac{\li(\mu+e^t)}{e^t} = \lim_{s \to 0}  \frac{\li(\mu+s)-\li(\mu)}{s} = \li'(\mu) = \frac{1}{\log \mu} = 2.684510350820\ldots.$$
Consequently, one has
$$0 \leq \frac{\li(e^t R_t)}{e^t} < \frac{\li(\mu+e^t)}{e^t} <\frac{1}{\log \mu}$$
for all $t \in \RR$.  Thus, the function $\beta(t,r_t)$ is positive and bounded.
More precise bounds on the main term $\frac{\li(e^t R_t)}{e^t} $ of $\beta(t,r_t)$ can be obtained from the following lemma, which follows readily from the fact that the integrand $\frac{1}{\log u}$ of $\li(x) = \int_0^x \frac{du}{\log u}$ is positive, decreasing, and concave up on $(1,\infty)$. 

\begin{lemma}\label{lem0a}
One has the following.
\begin{enumerate}
\item For all $y > x > 1$, one has
$$0< \frac{y-x}{\log \left(\frac{x+y}{2} \right)} < \li(y)-\li(x) < \frac{y-x}{\log x}.$$
\item For all $y > x > 0$ and all $t > -\log x$, one has
$$0 < \frac{y-x}{t + \log \left(\frac{x+y}{2} \right)}< \frac{\li(e^t y)}{e^t}-\frac{\li(e^t x)}{e^t} < \frac{y-x}{t + \log x}.$$
\end{enumerate}
\end{lemma}

\begin{corollary}\label{Rcor}
Let $t \in \RR$, let $r_t = \mu e^{-t}$,  and let $R_t   = \lceil r_t \rceil$.
One has 
$$0\leq  \beta_{R_t}(t, R_t) = \frac{\li(e^t R_t)}{e^t}    < \frac{\li(\mu +e^t)}{e^t} < \frac{1}{\log \mu}$$
and
$$0 \leq \frac{R_t-r_t}{t + \log R_t }  \leq \frac{R_t-r_t}{t + \log \left(\frac{R_t+r_t}{2} \right)} \leq \frac{\li(e^t R_t)}{e^t}  \leq  \frac{R_t-r_t}{\log \mu} < \frac{1}{\log \mu}.$$
\end{corollary}

The discussion above motivates the following definition.

\begin{definition}
Let $t \in \RR$.  
\begin{enumerate}
\item Let 
$$R_t   = \lceil \mu e^{-t} \rceil.$$  Equivalently, $R_t$ is the unique positive integer $N$ such that $t \in [\log  \mu - \log N, \log \mu - \log (N-1))$ (where we set $\log 0 = -\infty$, so that $R_t = 1$ if and only if $t \in [\log \mu, \infty)$).
\item Let $$\beta_x(t) = \beta_x(t, \mu e^{-t}) =  \frac{\li (e^t x)}{e^t}  -\sum_{R_t \leq k < x}   \frac{1}{H_k-\gamma+t}$$
for all $x \geq \mu e^{-t}$,
and let
$$\beta(t) = \beta(t,\mu e^{-t}) =  \lim_{n \to \infty} \beta_n(t) = \lim_{x \to \infty} \beta_x(t).$$
\end{enumerate}
\end{definition}

 A graph of the function $R_t$ is provided in Figure \ref{har3}, and a graph of the function $\beta_{R_t}(t) = \frac{\li(e^t R_t)}{e^t} \approx \beta(t)$ alongside a graph of its upper bound $ \frac{\li(\mu +e^t)}{e^t}$ as in Corollary \ref{Rcor} is provided in Figure \ref{har2}.

\begin{figure}[ht!]
\centering
\includegraphics[width=70mm]{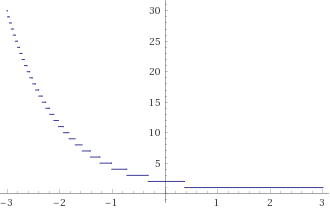}
\caption{Graph of $R_t = \lceil \mu e^{-t} \rceil$ on $[-3,3]$  \label{har3}}
\end{figure}

\begin{figure}[ht!]
\centering
\includegraphics[width=130mm]{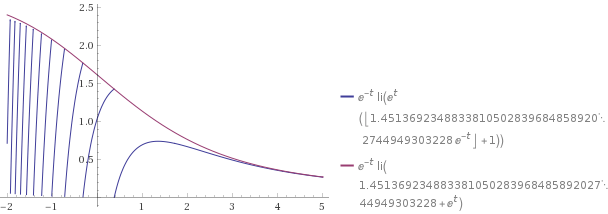}
\caption{Approximate graph of $\frac{\li(e^t R_t)}{e^t}$ and its upper bound  $ \frac{\li(\mu +e^t)}{e^t} < \frac{1}{\log \mu} $ on $[-2,5]$  \label{har2}}
\end{figure}

Applying Corollary \ref{maincor} to $r =  \mu e^{-t}$, we obtain the following.

\begin{corollary}\label{maincor2}
For all $t\in \RR$ and all $n  \geq  R_t = \lceil \mu e^{-t} \rceil$, one has
\begin{align*}
& \int_{n+1}^{\infty} \frac{dx}{24x^2 (t+\log x)^2} +
\frac{1}{24(n+1)(t+\log (n+1))^2} \\  & \qquad \leq \beta(t) - \beta_n(t)   \\ & \qquad \leq \int_n^\infty \frac{dx}{24x^2 (t+\log x)^2}  +\frac{1}{24(n+\sfrac{1}{2})(t+\log (n+\sfrac{1}{2}))^2} \\ & \qquad \qquad  + \frac{1}{24n^2(t+\log n)^2}  +  \frac{1}{12n^2(t+\log n)^3}.
\end{align*}
Moreover, one has
$$\beta(t) = \beta_{n}(t) + \frac{1+o(1)}{12n(\log n)^2} \ (n \to \infty).$$
\end{corollary}

It is clear that Corollary \ref{maincor2} implies Theorem \ref{maintheorem}.   Letting $n = R_t$,  we also conclude the following.

\begin{corollary}\label{maincor2b}
For all $t \in \RR$, one has
\begin{align*}
& \int_{R_t+1}^{\infty} \frac{dx}{24x^2 (t+\log x)^2} +
\frac{1}{24(R_t+1)(t+\log (R_t+1))^2} \\  & \qquad \leq  \beta(t) - \frac{\li(e^t R_t)}{e^t}   \\ &  \qquad \leq \int_{R_t}^\infty \frac{dx}{24x^2 (t+\log x)^2}  +\frac{1}{24(R_t+\sfrac{1}{2})(t+\log (R_t+\sfrac{1}{2}))^2} \\ & \qquad \qquad  + \frac{1}{24R_t^2(t+\log R_t)^2}  +  \frac{1}{12R_t^2(t+\log R_t)^3}.
\end{align*}
\end{corollary}

Table 1 shows the upper and lower bounds for $\beta(t)$ (rounded up and down, respectively) for all integers $t \in [-15,15]$ provided by the inequalities in Corollary \ref{maincor2b} above and compares them with approximate values of $\beta(t)$ that we computed using WolframAlpha by taking $n$ in the limit expression $\lim_{n \to \infty}\left( \frac{\li (e^t n)}{e^t}  - \sum_{k = \lceil \mu e^{-t} \rceil}^{n-1}  \frac{1}{H_k-\gamma+t} \right)$ for $\beta(t)$ as large as the online tool would allow.   Notice that the bounds on $\beta(t)$ thus computed for integers $t \in [-15, -7]$ are better than such direct estimates of $\beta(t)$.  All of these bounds on $\beta(t)$ can of course can be improved by increasing $n$ as in Corollary \ref{maincor2}.   For instance, in Table 2, we computed the bounds on $\beta(t)$ from Corollary \ref{maincor2} with $n = 50$ for nine special values of $t$ that of particular interest in the next three sections.

\begin{table}\label{table1}
\caption{Upper and lower bounds of $\beta(t)$ in Corollary \ref{maincor2b} (with $n = R_t$)} \bigskip
\footnotesize
\centering 
\begin{tabular}{r|l|l|l|l|l} 
$t$  & $\beta(t) <$   & $\beta(t) \approx$  &   $\beta(t) >$ & $\frac{\li(e^t R_t)}{e^t} \approx$ & $R_t$ \\ \hline 
$15$   & 0.07236490 & $0.07220$ & 0.07203360 & 0.07187354 &  $ 1$  \\
$14$ & 0.07805640  & $0.07786$  & 0.07767424 & 0.07749225 & $ 1$  \\ 
$13$  & 0.08473348 & $0.08450 $  & 0.08428780  & 0.08407903 & $ 1$ \\ 
$12$  & 0.09268293 & $ 0.09241 $  &  0.09215649 & 0.09191455 & $ 1$ \\ 
$11$ & 0.1023178 & $ 0.10199 $ &  0.1016865 & 0.1014028 & $ 1$ \\ 
$10$ & 0.1142551  & $ 0.11385$ &  0.1134844 & 0.1131470 & $ 1$ \\ 
$9$ &  0.1294542 & $ 0.12895$ &  0.1284922 & 0.1280844  & $ 1$ \\ 
$8$ & 0.1494680   & $0.14880 $ &   0.1482342 & 0.1477310 & $ 1$ \\ 
$7$ & 0.1769053  & $ 0.17601 $ & 0.1752663  & 0.1746297  & $ 1$ \\ 
$6$ & 0.2162592 & $ 0.21500$ &  0.2139786 & 0.2131473 & $ 1$ \\ 
$5$ & 0.2752827 & $ 0.27337$ &  0.2718986 &  0.2707663 & $ 1$ \\ 
$4$ & 0.3667085  & $ 0.36349$ &  0.3611866 & 0.3595520 & $ 1$ \\ 
$3$ & 0.5076564 & $ 0.50123$ &  0.4971471 & 0.4945764 & $ 1$ \\ 
$2$ & 0.7018947 & $0.68422$  &  0.6751323 & 0.6704827 & $ 1$ \\ 
$1$ & 0.8530561 & $0.74190$ &  0.7082072 &  0.6971749 & $ 1$  \\ 
$0$  &   1.1635319  & $1.09564$  &  1.0615462 & 1.0451638 & $ 2$ \\ 
$-1$  & 0.3044511  & $0.22299$  & 0.1512070  & 0.1443674 & $ 4$ \\ 
$-2$  & 0.7531859 & $0.74512$ & 0.7346875 & 0.7160222 & $ 11$  \\ 
$-3$  & 2.2054409  & $2.20450 $ & 2.2027692  & 2.1938815 &  $ 30$ \\ 
$-4$   & 2.0135694 & $2.01342$  & 2.0131210  & 2.0090540 &  $80$ \\ 
$-5$   & 1.6003036 & $ 1.60028$ & 1.6002378  & 1.5986089 &  $ 216$  \\ 
$-6$ & 1.2766878 & $ 1.27667$   & 1.2766786   & 1.2760617 &  $586$ \\ 
$-7$  & 1.0210154 & $1.02099$  &  1.0210141  & 1.0207849 &  $1592$  \\ 
$-8$   & 1.4207283 & $1.42068$  & 1.4207280  & 1.4206435 &  $4327$   \\ 
$-9$  & 1.1631179 & $1.16309$ & 1.1631178  &  1.16308670 & $11761$  \\ 
$-10$  & 1.2488004  & $1.24879$ & 1.2488003  & 1.24878890 &  $31969$  \\ 
$-11$  & 1.3764747 & $1.37647$ &  1.3764746 &  1.37647048 &  $86900$  \\
$-12$  & 1.8869487 & $1.88694$  & 1.8869486  & 1.88694713 & $236218$ \\
$-13$   & 2.1844846 & $2.18448$ & 2.1844845  & 2.18448397 &  $642106$ \\
$-14$  & 0.37643374 & $0.37643$  &  0.37643373 & 0.37643352 & $1745423$  \\
$-15$   &  2.032965028  & $2.032965$ &   2.032965027 & 2.032964950 & $4744552$  
\end{tabular}
\end{table}

\begin{table}\label{table2}
\caption{Upper and lower bounds of $\beta(t)$ in Corollary \ref{maincor2} with $n = 50$} \bigskip
\footnotesize
\centering 
\begin{tabular}{l|l|l|l|l} 
$t$  & $\beta(t) <$    &   $\beta(t) >$ & $\frac{\li(e^t R_t)}{e^t} $ & $R_t $ \\ \hline 
%$2$ &  0.6842208546 & 0.6842194424  &  $\li(e^2)/e^2 \approx 0.67048271$ & 1 \\
$\gamma+1\approx 1.577216$ & 0.7509547014   & 0.7509261228  & $\li(e^{\gamma+1})/e^{\gamma+1}\approx 0.730170$  & 1 \\
$\log \alpha \approx 1.347155$  & 0.7695247294 &  0.7695229079 & $\li(\alpha)/\alpha \approx 0.742305$ &  1 \\
1 & 0.7418976158 & 0.7418955006 & $\li(e)/e \approx 0.697175$ & 1 \\
$\log 2 \approx 0.693147$ &  0.6026096358  & 0.6026071971  &  $\li(2)/2 \approx 0.522582$  &  1 \\
$\gamma \approx 0.577216$ &   0.4986013304   & 0.4985987518  & $\li(e^\gamma)/e^\gamma \approx 0.393102$  & 1 \\
$\log \mu \approx 0.372507$  & 0.1952555336 & 0.1952526746 & 0 &  1 \\
 0 & 1.0956456993  &  1.0956421994    &  $\li(2) \approx 1.045164$ &  2 \\
 $-\log 2\approx -0.693147$ & 0.3417372460  &  0.3417318184 &   $2 \li(3/2) \approx 0.250130$ &  3 \\
$-1$  & 0.2229882714 & 0.2229814526 & $e \li(4/e) \approx 0.144367$ & 4 \\
\end{tabular}
\end{table}

Tables 1 and 2 also provide approximate values for the coarsest of all of our lower bounds of $\beta(t)$, namely, $\beta_{R_t}(t) = \frac{\li(e^t R_t)}{e^t}$.  In particular, one can see that
$\beta(t) \approx \frac{\li(e^t R_t)}{e^t}$ for $|t|$ sufficiently large.  This is made precise by the following proposition.

\begin{proposition}
Let $t \in \RR$. 
\begin{enumerate}
\item One has
$${R_t} \sim \mu e^{-t} \ (t \to -\infty),$$
$$t+\log R_t \geq \log \mu,$$
and $$\limsup_{t \to -\infty} \frac{\li(e^t R_t)}{e^t} = \lim_{t \to -\infty} \frac{1}{t + \log R_t} = \li'(\mu) =  \frac{1}{\log \mu} = 2.684510350820\ldots.$$
\item One has
$$\beta(t) = \frac{\li(e^t R_t)}{e^t} + \frac{(1+o(1)) e^t}{12\mu (\log \mu)^2} \ (t \to -\infty)$$
and
$$\limsup_{t \to -\infty} \beta(t) =  \frac{1}{\log \mu},$$
where also $\frac{1}{12\mu (\log \mu)^2} =-\frac{\li''(\mu) }{12} = \frac{1}{2.416734786403\ldots}$.
\item One has
$$\beta(t) < \frac{1}{24} \mu \li\left(\frac{1}{\mu}\right) + \frac{1}{24 \log \mu} + \frac{1}{ 24 (\log \mu)^2} + \frac{1}{36 (\log(\sfrac{3\mu}{2}))^2}+\frac{1}{12(\log \mu)^3} = 2.0248039 \ldots$$
on $[\log \mu, \infty)$.
\item  One has
$$\beta(t)  =  \frac{\li(e^t)}{e^t} + O\left(\frac{1}{t^2}\right) =  \frac{1}{t} + O\left(\frac{1}{t^2}\right)  \ (t \to \infty)$$
and 
 $$\beta(t)  \sim  \frac{\li(e^t)}{e^t} \sim  \frac{1}{t}  \ (t \to \infty).$$
\item One has
$$\beta(t) <  \limsup_{u \to -\infty} \beta(u) =\frac{1}{\log \mu},$$
and therefore the least upper bound of the range of $\beta$ is  $\frac{1}{\log \mu}$.
\end{enumerate}
\end{proposition}

\begin{proof}
Statement (1) is an easy consquence of Corollary \ref{Rcor}.  Moreover, by Corollary \ref{maincor2} and Lemma \ref{lambda}, one has
\begin{align}\label{hh}
 & \frac{1}{12(R_t+1)(t+\log (R_t+1))^2} - \frac{1}{12(R_t+1)^2(t+\log( R_t+1))^3} \nonumber  \\ &  \qquad < \beta(t) - \frac{\li(e^t R_t)}{e^t}  \nonumber  \\
 & \qquad  < \frac{1}{12R_t(t+\log R_t)^2} + \frac{1}{24R_t^2(t+\log R_t)^2} + \frac{1}{12R_t^2(t+\log R_t)^3}  \nonumber \\
  & \qquad \leq  \frac{e^t }{12 \mu (\log \mu)^2} +  \frac{e^{2t} }{24\mu^2(\log \mu)^2}\left(1+\frac{2}{\log \mu} \right),
\end{align} 
and, since $$ \frac{1}{12(R_t+1)(t+\log (R_t+1))^2}  \sim \frac{1}{12R_t(t+\log R_t)^2}  \sim  \frac{e^t }{12 \mu (\log \mu)^2} \ (t \to -\infty),$$  statement (2) follows.  By Corollary \ref{maincor2}, one has
$$\beta(t) < \frac{\li(e^t)}{e^t} + \frac{1}{24} e^t \li\left(\frac{1}{e^{t}}\right) + \frac{1}{24 t} + \frac{1}{ 24 t^2} + \frac{1}{36 (t+\log( \sfrac{3}{2}))^2}+\frac{1}{12t^3}$$
on $[\log \mu, \infty)$.  Moreover, the upper bound of $\beta(t)$ above has negative derivative, and is therefore decreasing, on $[\log \mu, \infty)$.  Statement (3) follows.   Statement (4) follows Eq.\ (\ref{hh}) and the fact that $R_t = 1$ on $[\log \mu, \infty)$.

Finally, we prove statement (5).  By  (3), we know that $\beta(t) < 2.0248039  \ldots < \frac{1}{\log \mu}$ on $[\log \mu,\infty)$.   Moreover, by Corollary \ref{Rcor}, one has $\frac{\li(e^t R_t)}{e^t} < \frac{\li(\mu+e^t)}{e^t}.$   Therefore,  by Eq.\ (\ref{hh}), one has
$$\beta(t) < \frac{\li(\mu+e^t)}{e^t} + \frac{e^t }{12 \mu (\log \mu)^2} +  \frac{e^{2t} }{24\mu^2(\log \mu)^2}\left(1+\frac{2}{\log \mu} \right).$$
Moreover, the upper bound of $\beta(t)$ above is decreasing on $(\infty,-1]$ with limit $\frac{1}{\log \mu}$ as $t \to -\infty$, and it is less than $\frac{1}{\log \mu}$ also on $[-1,\log(\mu/2)]$.  Therefore one has $\beta(t) < \frac{1}{\log \mu}$  on $(-\infty,\log(\mu/2)]$.
Moreover, on $[\log (\mu/2),\log \mu)$, one has $R_t = 2$ and therefore, by Corollary \ref{maincor2},
$$\beta(t) < \frac{ \li(2e^t)}{e^t} +\frac{1}{24(t+\log 2)^2}+\frac{1}{60(t+\log (5/2))^2}+\frac{1}{96(t+\log 2)^2}+\frac{1}{48(t+\log 2)^3}.$$
Finally, the upper bound of $\beta(t)$ above  is maximized on $[\log (\mu/2),\log \mu]$ at the endpoint $\log \mu$, and therefore
$$\beta(t) < \frac{\li(2\mu)}{\mu} +\frac{1}{24(\log (2\mu))^2}+\frac{1}{60(\log (\sfrac{5\mu}{2}))^2}+\frac{1}{96(\log (2\mu))^2}+\frac{1}{48(\log (2\mu))^3}  = 1.501\ldots < \frac{1}{\log \mu}$$
on $[\log (\mu/2),\log \mu)$.  Thus, we have shown that $\beta(t) < \frac{1}{\log \mu}$ for all $t \in \RR$.
\end{proof}

%The analysis above explains why the estimate $\beta(t) \approx \frac{\li(e^t R_t)}{e^t}$ is much better for negative $t$ than for positive $t$.  To obtain better estimates of $\beta(t)$ for any $t$ (positive or negative), one can take larger values of $n$ in Corollary \ref{maincor2}.  For large negative values of $t$ this is achieved automatically, since $R_t \sim \mu e^{-t} \ (t \to \infty)$ grows exponentially as $t \to -\infty$.  Therefore, for large negative $t$, increasing $n$ slightly won't improve the approximation much, but for positive $t$, increasing $n$ slightly can make a big difference.  

\section{Approximating $\pi(x)$ with harmonic numbers}

The following result of Montgomery and Vaughan \cite{mv} is an analogue of Lemma \ref{lem0a} for the prime counting function.

\begin{theorem}[{\cite{mv}}]\label{mv}
For all $y > x > 0$ with $y - x > 1$, one has
$$0 \leq \pi(y) -\pi(x) < \frac{2 (y-x)}{\log (y-x)}.$$
\end{theorem}

The following is an immediate corollary of the result above.

\begin{corollary}\label{Ohelp3}
Let $N$ be a positive integer.  One has the following.
\begin{enumerate}
\item For all $y > x > 0$ and all  $t > -\log (y-x)$, one has
$$0 \leq \frac{\pi(e^t y)}{e^t}-\frac{\pi(e^t x)}{e^t} < \frac{2(y-x)}{t + \log (y-x)}.$$
\item For all $x \geq 1 $  and all $t > -\log N$, one has
$$0 \leq \frac{ \pi(e^t x)}{e^t}-\frac{\pi(e^t \lfloor x \rfloor)}{e^t} < \frac{2N}{t +\log N}.$$
\end{enumerate}
\end{corollary}

The following theorem describes the relevance of the function $\beta(t,r)$ (and  thus also the function $\beta(t)$) to the prime counting function.

\begin{theorem}\label{RH1}
Let $\delta$ denote the infimum of the real parts of the zeros of the Riemann zeta function.   Suppose that $M> 0$, $C \geq \mu$, and $\alpha$ are constants such that $|\pi(x) - \li(x) | \leq M x^\alpha \log x$ for all $x \geq C$.  Then $\alpha \geq \delta$ and, for all $t,r,\lambda \in \RR$ with $\lceil r \rceil  \geq \mu e^{-t}$ and $\lambda \geq \beta(t,r)$, and for all integers $n \geq Ce^{-t}$, one has the following.
\begin{enumerate}
\item $\displaystyle \left|\pi (e^t n)-e^t\sum_{r \leq k < n}  \frac{1}{H_k-\gamma+t}  \right| <  M e^{\alpha t} n^\alpha (t+\log n)+ \lambda e^t$.
\item $\displaystyle \left|\pi (e^t n)-e^t\sum_{r \leq k < n}  \frac{1}{H_k-\gamma+t}  \right| <  M e^{\alpha t} n^\alpha (H_n-\gamma+t)+ \lambda e^t$.
\item $\displaystyle \left|\PP(e^t n) -\frac{1}{n}\sum_{r \leq k < n}  \frac{1}{H_k-\gamma+t}   \right| <  \frac{M}{e^{(1-\alpha )t}} \frac{t+\log n}{n^{1-\alpha} }+ \frac{\lambda}{n}$.
\item $\displaystyle \left|\PP(e^t n) -\frac{1}{n}\sum_{r \leq k  < n}  \frac{1}{H_k-\gamma+t}   \right| <  \frac{M}{e^{(1-\alpha )t}} \frac{H_n-\gamma+t}{n^{1-\alpha}} + \frac{\lambda}{n}$.
\end{enumerate}
Conversely, if any of the conditions above hold for all $n \geq C$ for some constants $M > 0$, $C \geq \mu$, $\lambda > 0$, and $t,r, \alpha \in \RR$, then $\alpha \geq \delta$, so there exist constants $M'> 0$ and $C' \geq \mu$ such that $|\pi(x) - \li(x) | \leq M' x^\alpha \log x$
for all $x \geq C'$.
\end{theorem}

\begin{proof}
Let $\alpha$, $M$, $C$, $t$, $r$, and $\lambda$ be given as in the forward hypothesis.
By Corollary \ref{maincor}, for all $n \geq \lceil r \rceil \geq \mu e^{-t}$, one has
$$0 \leq \frac{\li (e^t n)}{e^t}  - \sum_{r \leq k < n}   \frac{1}{H_k-\gamma+t} = \beta_n(t, r) < \beta(t,r).$$
Moreover, by hypothesis and the obvious change of variables one has
$$\left |\frac{\pi (e^t n)}{e^t}  -\frac{\li (e^t n)}{e^t} \right | \leq \frac{M}{e^{(1-\alpha )t}} n^\alpha (t+\log n)$$
for all $n \geq Ce^{-t} \geq \mu e^{-t}$.  Therefore, by the triangle inequality, one has
 $$\left|\frac{\pi (e^t n)}{e^t}-\sum_{r \leq k < n}  \frac{1}{H_k-\gamma+t}  \right| <  \frac{M}{e^{(1-\alpha )t}} n^\alpha (t+\log n)+ \beta(t,r)$$
for all $n \geq Ce^{-t}$, while also $t+\log n < H_n-\gamma +t$ for all $n$.  The forward direction of the theorem follows.

Conversely, suppose that  $\alpha$, $M$, $C$, $t$, $r$, and $\lambda$ are constants satisfying any of the hypotheses (1)--(4).  We wish to show that $\alpha \geq \delta$.  We may suppose without loss of generality that $\alpha > 0$.   Now, since $\log n \sim t+\log n \sim H_n - \gamma+t \ (n \to \infty)$, any one of statements (1)--(4) implies that
$$\frac{\pi(e^t n)}{e^t}  - \sum_{r \leq k < n}  \frac{1}{H_k-\gamma+t}  = O( n^\alpha \log n) \ (n \to \infty),$$
where the $O$ constant depends on the given constants.   Moreover,  by Corollary \ref{maincor} one has
$$\frac{\li (e^t n)}{e^t}  -\sum_{r \leq k < n}    \frac{1}{H_k-\gamma+t} = O(1) \ (n \to \infty),$$
so it follows that
$$\frac{\pi(e^t n)}{e^t}  - \frac{\li (e^t n)}{e^t}  = O( n^\alpha\log n) \ (n \to \infty).$$
We wish to show that we can replace the discrete variable $n$ in the above estimate with a continuous variable $x$.  Choose any positive integer $N > e^{-t}$, so that $t > - \log N$.   Then, by Corollary \ref{Ohelp3}, for all $x > 1$ one has
$$0 \leq \frac{ \pi(e^t x)}{e^t}-\frac{\pi(e^t \lfloor x \rfloor)}{e^t} < \frac{2N}{t +\log N},$$
so that 
$$\frac{ \pi(e^t x)}{e^t}-\frac{\pi(e^t \lfloor x \rfloor)}{e^t} = O(1) \ (x \to \infty).$$
Similarly, by Lemma \ref{lem0a}, one has
$$\frac{ \li(e^t x)}{e^t}-\frac{\li(e^t \lfloor x \rfloor)}{e^t} = O(1) \ (x \to \infty).$$
Therefore, one has
\begin{align*}
\frac{\pi(e^t x)}{e^t}  - \frac{\li (e^t x)}{e^t}  & =   \left(\frac{\pi(e^t\lfloor  x \rfloor )}{e^t}- \frac{\li(e^t\lfloor  x \rfloor )}{e^t}\right) + \left( \frac{\pi(e^t x)}{e^t} -  \frac{\pi(e^t \lfloor x \rfloor )}{e^t} \right)- \left (\frac{\li (e^t x)}{e^t}  - \frac{\li(e^t\lfloor  x \rfloor )}{e^t}\right) \\ &  = O( \lfloor x \rfloor^\alpha \log \lfloor x \rfloor) \ (x \to \infty) \\
&  = O( x^\alpha \log  x ) \ (x \to \infty)
\end{align*}
It follows, then, that $\pi( x)  - \li (x)  = O( x^\alpha\log x) \ (x \to \infty),$
and therefore $\alpha \geq \delta$.
\end{proof}

The special case where $r = \mu e^{-t}$ yields the following.

\begin{corollary}\label{RH2}
Let $\delta$ denote the infimum of the real parts of the zeros of the Riemann zeta function.   Suppose that $M> 0$, $C \geq \mu$, and $\alpha$ are constants such that $|\pi(x) - \li(x) | \leq M x^\alpha \log x$ for all $x \geq C$.  Then $\alpha \geq \delta$ and, for all $t,\lambda \in \RR$ with $\lambda \geq \beta(t)$ (e.g., $\lambda = \frac{1}{\log \mu}$), and for all integers $n \geq Ce^{-t}$, one has the following.
\begin{enumerate}
\item $\displaystyle \left|\pi (e^t n)-e^t\sum_{k = \lceil \mu e^{-t} \rceil}^{n-1}  \frac{1}{H_k-\gamma+t}  \right| <  M e^{\alpha t} n^\alpha (t+\log n)+  \lambda e^t$.
\item $\displaystyle \left|\pi (e^t n)-e^t\sum_{k = \lceil \mu e^{-t} \rceil}^{n-1}  \frac{1}{H_k-\gamma+t}  \right| <  M e^{\alpha t} n^\alpha (H_n-\gamma+t)+ \lambda e^t$.
\item $\displaystyle \left|\PP(e^t n) -\frac{1}{n}\sum_{k = \lceil \mu e^{-t} \rceil}^{n-1}  \frac{1}{H_k-\gamma+t}   \right| <  \frac{M}{e^{(1-\alpha )t}} \frac{t+\log n}{n^{1-\alpha} }+ \frac{\lambda}{n}$.
\item $\displaystyle \left|\PP(e^t n) -\frac{1}{n} \sum_{k = \lceil \mu e^{-t} \rceil}^{n-1}  \frac{1}{H_k-\gamma+t}   \right| <  \frac{M}{e^{(1-\alpha )t}} \frac{H_n-\gamma+t}{n^{1-\alpha}}+\frac{\lambda}{n}$.
\end{enumerate}
Conversely, if any of the conditions above hold for all $n \geq C$ for some constants $M > 0$, $C \geq \mu$, $\lambda > 0$, and $t,r, \alpha \in \RR$, then $\alpha \geq \delta$, so there exist constants $M'> 0$ and $C' \geq \mu$ such that
\begin{align*}
|\pi(x) - \li(x) | \leq M' x^\alpha \log x
\end{align*}
for all $x \geq C'$.
\end{corollary}

\section{Riemann hypothesis equivalents using harmonic numbers}

In 1976 \cite{sch}, L.\ Schoenfeld proved  that the Riemann hypothesis is equivalent to    Eq.\ (\ref{sch}).
From Schoenfeld's result, we can relate Theorem \ref{RH1} to the Riemann hypothesis as follows.

\begin{theorem}\label{RH1aa}
Suppose that the Riemann hypothesis implies that
\begin{align*}
|\pi(x)- \li(x)| \leq M\sqrt{x} \log x
\end{align*}
for all $x \geq C$, for constants $M > 0$ and $C \geq \mu$.  For example, this implication holds for $M = \frac{1}{8\pi}$ and $C = 2657$.    Let $t,r,\lambda \in \RR$ with $r > e^{-t}$ and $\lambda \geq \beta(t,r)$.     Then the Riemann hypothesis holds if and only if any of the following conditions hold for all $n \geq Ce^{-t}$.
\begin{enumerate}
\item $\displaystyle \left|\pi (e^t n)-e^t\sum_{r \leq k < n}  \frac{1}{H_k-\gamma+t}  \right| \leq  M e^{t/2}\sqrt{n} (t+\log n) + \lambda e^{t}$.
\item $\displaystyle \left|\pi (e^t n)-e^t\sum_{r \leq k < n}  \frac{1}{H_k-\gamma+t}  \right| \leq  Me^{t/2} \sqrt{n} (H_n-\gamma+t) + \lambda e^t$.
\item $\displaystyle \left|\PP(e^t n) -\frac{1}{n}\sum_{r \leq k < n}  \frac{1}{H_k-\gamma+t}   \right| \leq \frac{M}{e^{t/2}} \frac{t+\log n}{\sqrt{n}} + \frac{\lambda}{n}$.
\item $\displaystyle \left|\PP(e^t n) -\frac{1}{n}\sum_{r \leq k  < n}  \frac{1}{H_k-\gamma+t}   \right| \leq  \frac{M}{e^{t/2}} \frac{H_n-\gamma+t}{\sqrt{n}}+ \frac{\lambda}{n}$.
\end{enumerate}
\end{theorem}

\begin{corollary}\label{RH1a}
Suppose that the Riemann hypothesis implies that
\begin{align*}
|\pi(x)- \li(x)| \leq M\sqrt{x} \log x
\end{align*}
for all $x \geq C$, for constants $M > 0$ and $C \geq \mu$.  For example, this implication holds for $M = \frac{1}{8\pi}$ and $C = 2657$.    Let $t,\lambda \in \RR$ with $\lambda \geq \beta(t)$ (e.g., $\lambda = \frac{1}{\log \mu}$).     Then the Riemann hypothesis holds if and only if any of the following conditions hold for all $n \geq Ce^{-t}$.
\begin{enumerate}
\item $\displaystyle \left|\pi (e^t n)-e^t\sum_{r \leq k < n}  \frac{1}{H_k-\gamma+t}  \right| \leq  M e^{t/2}\sqrt{n} (t+\log n) + \lambda e^{t}$.
\item $\displaystyle \left|\pi (e^t n)-e^t\sum_{r \leq k < n}  \frac{1}{H_k-\gamma+t}  \right| \leq  Me^{t/2} \sqrt{n} (H_n-\gamma+t) + \lambda e^t$.
\item $\displaystyle \left|\PP(e^t n) -\frac{1}{n}\sum_{r \leq k < n}  \frac{1}{H_k-\gamma+t}   \right| \leq \frac{M}{e^{t/2}} \frac{t+\log n}{\sqrt{n}} + \frac{\lambda}{n}$.
\item $\displaystyle \left|\PP(e^t n) -\frac{1}{n}\sum_{r \leq k  < n}  \frac{1}{H_k-\gamma+t}   \right| \leq  \frac{M}{e^{t/2}} \frac{H_n-\gamma+t}{\sqrt{n}}+ \frac{\lambda}{n}$.
\end{enumerate}
\end{corollary} 

Theorem \ref{RH1aa} and Corollary \ref{RH1a} can be used to yield a number of arithmetical equivalences of the Riemann hypothesis.  For example, letting $t$ equal $\gamma$, $\gamma+1$, $0$, $\log 2$, and $1$, respectively,  and employing the upper and lower bounds for $\beta(t)$ provided in Table 2, we obtain the several  Riemann hypothesis equivalents listed in Corollaries \ref{RH5} through \ref{RH9} below.  (For the sake of brevity we express them all in terms of $\PP(x)$ instead of $\pi(x)$.)

\begin{corollary}\label{RH5}
Let $\lambda \geq \beta(\gamma)$, e.g., $\lambda = 0.4986013304$.
The Riemann hypothesis holds if and only if any of the following equivalent conditions hold for all  integers $n \geq 803$.
\begin{enumerate}
%\item $\displaystyle \left|\pi (e^\gamma  n)-e^\gamma \sum_{k = 1}^{n-1}  \frac{1}{H_k}  \right| <  \frac{e^{\gamma/2}}{8 \pi}  \sqrt{n} (\log n+\gamma)+ \lambda e^\gamma $.
%\item $\displaystyle \left|\pi (e^\gamma n)-e^\gamma\sum_{k = 1}^{n-1}  \frac{1}{H_k}  \right| < \frac{e^{\gamma/2}}{8 \pi}   \sqrt{n} H_n+ \lambda e^\gamma $.
\item $\displaystyle \left|\PP(e^\gamma n) -\frac{1}{n} \sum_{k = 1}^{n-1}   \frac{1}{H_k}   \right| < \frac{1}{8 \pi e^{\gamma/2}}  \frac{\log n+\gamma}{\sqrt{n}}+ \frac{\lambda}{n}$.
\item $\displaystyle \left|\PP(e^\gamma n) -\frac{1}{n} \sum_{k = 1}^{n-1}   \frac{1}{H_k}   \right| < \frac{1}{8 \pi e^{\gamma/2}} \frac{H_n}{\sqrt{n}}+\frac{\lambda}{n}$.
\end{enumerate}
\end{corollary} 

\begin{proof}
By Corollary \ref{RH1a}, the Riemann hypothesis holds if and only if any of the conditions hold for all $n \geq 2657 e^{-\gamma} > 1491$, and the given inequalities can be verified directly to hold for all $803 \leq n \leq 1491$, even for the lower bound $\lambda = 0.49859$ of $\beta(\gamma)$. 
\end{proof}

\begin{corollary}\label{RH6}
Let $\lambda \geq \beta(\gamma)+1 = \beta(\gamma, \mu e^{-\gamma} + 1)$, e.g. $\lambda = 1.4986013304$.
The Riemann hypothesis holds if and only if any of the following equivalent conditions hold for all  positive integers $n$.
\begin{enumerate}
%\item $\displaystyle \left|\pi (e^\gamma  n)-e^\gamma \sum_{k = 2}^{n-1}  \frac{1}{H_k}  \right| <  \frac{e^{\gamma/2} }{8 \pi} \sqrt{n} (\log n+\gamma)+ \lambda e^\gamma $.
%\item $\displaystyle \left|\pi (e^\gamma n)-e^\gamma\sum_{k = 2}^{n-1}  \frac{1}{H_k}  \right| < \frac{e^{\gamma/2} }{8 \pi}  \sqrt{n} H_n+ \lambda e^\gamma $.
\item $\displaystyle \left|\PP(e^\gamma n) -\frac{1}{n} \sum_{k = 2}^{n-1}   \frac{1}{H_k}   \right| < \frac{1}{8 \pi e^{\gamma/2}}  \frac{\log n+\gamma}{\sqrt{n}}+ \frac{\lambda}{n}$.
\item $\displaystyle \left|\PP(e^\gamma n) -\frac{1}{n} \sum_{k = 2}^{n-1}   \frac{1}{H_k}   \right| < \frac{1}{8 \pi e^{\gamma/2}} \frac{H_n}{\sqrt{n}}+\frac{\lambda}{n}$.
\end{enumerate}
\end{corollary}

\begin{proof}
By Theorem \ref{RH1aa}, the Riemann hypothesis holds if and only if any of the conditions hold for all $n \geq 2657 e^{-\gamma} > 1491$, and the given inequalities can be verified directly to hold for all $1 \leq n \leq 1491$, even for the lower bound $\lambda = 1.49859$ of $\beta(\gamma) + 1$.
\end{proof}

\begin{corollary}\label{RH6b}
Let $\lambda \geq \beta(\gamma+1)$, e.g. $\lambda = 
0.7509547014$.
The Riemann hypothesis holds if and only if any of the following equivalent conditions hold for all positive integers $n$.
\begin{enumerate}
%\item $\displaystyle \left|\pi (e^{\gamma+1}  n)-e^{\gamma+1} \sum_{k = 1}^{n-1}  \frac{1}{H_k+1}  \right| <  \frac{e^{(\gamma+1)/2} }{8 \pi} \sqrt{n} (\log n+\gamma+1)+ \lambda e^{\gamma +1}$.
%\item $\displaystyle \left|\pi (e^{\gamma+1} n)-e^{\gamma+1}\sum_{k = 1}^{n-1}  \frac{1}{H_k+1}  \right| < \frac{e^{(\gamma+1)/2} }{8 \pi}  \sqrt{n} (H_n+1)+ \lambda e^{\gamma+1} $.
\item $\displaystyle \left|\PP(e^{\gamma+1} n) -\frac{1}{n} \sum_{k = 1}^{n-1}   \frac{1}{H_k+1}   \right| < \frac{1}{8 \pi e^{(\gamma+1)/2}}  \frac{\log n+\gamma+1}{\sqrt{n}}+ \frac{\lambda}{n}$.
\item $\displaystyle \left|\PP(e^{\gamma+1} n) -\frac{1}{n} \sum_{k = 1}^{n-1}   \frac{1}{H_k+1}   \right| < \frac{1}{8 \pi e^{(\gamma+1)/2}} \frac{H_n+1}{\sqrt{n}}+\frac{\lambda}{n}$.
\end{enumerate}
\end{corollary}

\begin{proof}
By Corollary \ref{RH1a}, the Riemann hypothesis holds if and only if any of the conditions hold for all $n \geq 2657 e^{-\gamma-1} > 548$, and the given inequalities can be verified directly to hold for all $1 \leq n \leq 548$, even for the lower bound $\lambda =   0.750926$ of $\beta(\gamma)$. 
\end{proof}

\begin{corollary}\label{RH7}
Let $\lambda \geq \beta(0)$, e.g., $\lambda  =  1.0956456993$.   The Riemann hypothesis holds if and only if any of the following equivalent conditions hold for all  positive integers $n \geq 1427$.
\begin{enumerate}
%\item $\displaystyle \left|\pi (n)- \sum_{k = 2}^{n-1}  \frac{1}{H_k-\gamma}  \right| <  \frac{1}{8 \pi}  \sqrt{n} \log  n+ \lambda $.
%\item $\displaystyle \left|\pi ( n)-\sum_{k = 2}^{n-1}  \frac{1}{H_k-\gamma}  \right| < \frac{1}{8 \pi}   \sqrt{n} (H_n-\gamma)+ \lambda $.
\item $\displaystyle \left|\PP(n) -\frac{1}{n}\sum_{k = 2}^{n-1}    \frac{1}{H_k-\gamma}   \right| < \frac{1}{8 \pi }  \frac{\log  n}{\sqrt{n}}+ \frac{\lambda}{n}$.
\item $\displaystyle \left|\PP(n) -\frac{1}{n}\sum_{k = 2}^{n-1}    \frac{1}{H_k-\gamma}   \right| < \frac{1}{8 \pi } \frac{H_n-\gamma}{\sqrt{n}}+\frac{\lambda}{n}$.
\end{enumerate}
\end{corollary}

\begin{proof}
By Corollary \ref{RH1a}, the Riemann hypothesis holds if and only if any of the conditions hold for all $n \geq 2657$, and the given inequalities can be verified directly to hold for all $1427 \leq n \leq 2657$, even for the lower bound $\lambda =  1.09564$ of $\beta(\gamma)$. 
\end{proof}

\begin{corollary}\label{RH8a}
Let $\lambda  \geq \beta(\log 2)$, e.g., $\lambda =   0.6026096358$.   The Riemann hypothesis holds if and only if any of the following equivalent conditions hold for all  positive integers $n \geq 714$.
\begin{enumerate}
%\item $\displaystyle \left|\pi (2n)- 2\sum_{k = 1}^{n-1}  \frac{1}{H_k-\gamma+\log 2}  \right| <  \frac{1}{8 \pi}  \sqrt{2n} \log (2n)+ 2\lambda $.
%\item $\displaystyle \left|\pi ( 2n)-2\sum_{k = 1}^{n-1}  \frac{1}{H_k-\gamma+\log 2}  \right| < \frac{1}{8 \pi}   \sqrt{2n} (H_n-\gamma+\log 2)+ 2\lambda $.
\item $\displaystyle \left|\PP(2n) -\frac{1}{n}\sum_{k = 1}^{n-1}    \frac{1}{H_k-\gamma+\log 2}   \right| < \frac{1}{8 \pi }  \frac{\log (2n)}{\sqrt{2n}}+ \frac{\lambda}{n}$.
\item $\displaystyle \left|\PP(2n) -\frac{1}{n}\sum_{k = 1}^{n-1}    \frac{1}{H_k-\gamma+\log 2}   \right| < \frac{1}{8 \pi } \frac{H_n-\gamma+\log 2}{\sqrt{2n}}+\frac{\lambda}{n}$.
\end{enumerate}
\end{corollary}

\begin{proof}
By Corollary \ref{RH1a}, the Riemann hypothesis holds if and only if any of the conditions hold for all $n \geq 2657/2 > 1328$, and the given inequalities can be verified directly to hold for all $714\leq n \leq 1328$, even for the lower bound $\lambda =  0.602607$ of $\beta(\gamma)$. 
\end{proof}

\begin{corollary}\label{RH9}
Let $\lambda  \geq \beta(1)$, e.g., $\lambda =  0.7418976158$.   The Riemann hypothesis holds if and only if any of the following equivalent conditions hold for all  positive integers $n \neq 82$.
\begin{enumerate}
%\item $\displaystyle \left|\pi (en)- e\sum_{k = 1}^{n-1}  \frac{1}{H_k-\gamma+1}  \right| <  \frac{1}{8 \pi}  \sqrt{en} \log (en)+ e\lambda$.
%\item $\displaystyle \left|\pi (e n)-e\sum_{k = 1}^{n-1}  \frac{1}{H_k-\gamma+1}  \right| < \frac{1}{8 \pi}   \sqrt{en} (H_n-\gamma+1)+ e\lambda$.
\item $\displaystyle \left|\PP(en) -\frac{1}{n}\sum_{k = 1}^{n-1}    \frac{1}{H_k-\gamma+1}   \right| < \frac{1}{8 \pi }  \frac{\log (en)}{\sqrt{en}}+ \frac{\lambda}{n}$.
\item $\displaystyle \left|\PP(en) -\frac{1}{n}\sum_{k = 1}^{n-1}    \frac{1}{H_k-\gamma+1}   \right| < \frac{1}{8 \pi } \frac{H_n-\gamma+1}{\sqrt{en}}+\frac{\lambda}{n}$.
\end{enumerate}
\end{corollary}

\begin{proof}
By Corollary \ref{RH1a}, the Riemann hypothesis holds if and only if any of the conditions hold for all $n \geq 2657e^{-1} > 977$, and the given inequalities can be verified directly to hold for all $1 \leq n \leq 977$ with $n \neq 82$, even for the lower bound $\lambda =  0.741895$ of $\beta(\gamma)$. 
\end{proof}

Curiously, $n= 82$ is the only (hypothetical) exception to the two Riemann hypothesis equivalents in Corollary \ref{RH9}.

Our results specialize to $O$ bounds as follows.

\begin{theorem}\label{Omain}
Let $\delta$ be the supremum of the real parts of the zeros of the Riemann zeta function, let $N$ be a positive integer, and let $t \in \RR$ so that $t \neq \gamma -H_n$ for all $n \geq N$ (which holds if $t > \gamma-H_N$).  Then $\delta$ is the smallest real number $\alpha$ such that $$\pi(e^{t} n) = e^{t} \sum_{k = N}^{n-1} \frac{1}{H_k -\gamma+t}  +  O\left(       n^\alpha H_n\right) \ (n \to \infty).$$
In particular, the Riemann hypothesis is equivalent to
$$\pi(e^{t} n) = e^{t} \sum_{k = N}^{n-1} \frac{1}{H_k -\gamma+t}  +  O\left(\sqrt{n}H_n\right) \ (n \to \infty).$$
\end{theorem}

\begin{corollary}\label{rh1A}
The Riemann hypothesis is equivalent to
$$ {\PP(e^{\gamma}n)} =  \frac{1}{n}\sum_{n = 1}^{n-1} \frac{1}{H_k}  + O \left(\frac{H_n}{\sqrt{n}} \right) \ (n \to \infty)$$
and  to
$$\PP(n) =  \frac{1}{n} \sum_{k = 1}^{n-1} \frac{1}{H_k -\gamma}  + O \left(\frac{H_n}{\sqrt{n}} \right) \ (n \to \infty).$$
\end{corollary}

For any positive real numbers $x_1, x_2, \ldots, x_n$, let $M_{-1}(x_1,x_2, \ldots,x_n) = \frac{n
}{x_1^{-1} + x_2^{-1}+\cdots+x_n^{-1}}$ denote the harmonic mean of $x_1, x_2, \ldots, x_n$.  Thus, for example, one has $M_{-1}(1,2,3,\ldots,n) = \frac{n}{H_n}$ for all positive integers $n$.  Since
$ \frac{1}{H_n-\gamma+t} = o(1) \ (n \to \infty),$  one can replace the upper limits $x-1$ and $n-1$ of the sums in Theorem \ref{Omain} and Corollary \ref{rh1A} with $x$ and $n$, respectively.
Thus, we have the following.

\begin{corollary}\label{what}
Let $t > -1$.  Each of the following statements is equivalent to the Riemann hypothesis.
\begin{enumerate}
\item $\displaystyle \PP(e^{t+\gamma}n) =  \frac{1}{M_{-1}(H_1+t,H_2+t,\ldots, H_n+t)}  + O \left(\frac{H_n}{\sqrt{n}} \right) \ (n \to \infty)$.
\item $\displaystyle \frac{1}{\PP(e^{t+\gamma}n)} =  M_{-1}(H_1+t,H_2+t,\ldots, H_n+t)  + O \left(\frac{H_n^3}{\sqrt{n}} \right) \ (n \to \infty).$
\item $\displaystyle \PP(e^{\gamma}n) =  \frac{1}{M_{-1}(H_1,H_2,\ldots, H_n)}  + O \left(\frac{H_n}{\sqrt{n}} \right) \ (n \to \infty)$.
\item $\displaystyle \frac{1}{\PP(e^{\gamma}n)} =  M_{-1}(H_1,H_2,\ldots, H_n)  + O \left(\frac{H_n^3}{\sqrt{n}} \right) \ (n \to \infty)$.
\item  $\displaystyle \PP(n) =  \frac{1}{M_{-1}(H_1-\gamma,H_2-\gamma,\ldots, H_n-\gamma)}  + O \left(\frac{H_n}{\sqrt{n}} \right) \ (n \to \infty)$.
\item  $\displaystyle \frac{1}{ \PP(n)} =  M_{-1}(H_1-\gamma,H_2-\gamma,\ldots, H_n-\gamma)  + O \left(\frac{H_n^3}{\sqrt{n}} \right) \ (n \to \infty)$.
\end{enumerate}
\end{corollary}

\section{Monotonicity properties of the error term $\beta(t)$}

In this final section, we examine the intervals of increase and decrease of the function $\beta(t) = \lim_{n \to \infty}\left( \frac{\li (e^t n)}{e^t}  - \sum_{k = \lceil \mu e^{-t} \rceil}^{n-1}  \frac{1}{H_k-\gamma+t} \right)$.

By Proposition \ref{monotone}, the function $\theta(t,1)$ is  strictly totally monotone on the interval $(0, \infty)$.  Since $\theta(t,1) = \beta(t) -\frac{\li(e^t)}{e^t}$ on $[\log \mu, \infty)$, it follows that $\beta(t)-\frac{\li(e^t)}{e^t}$ is  strictly totally monotone on the interval $[\log \mu, \infty)$.    Let $\alpha = 3.846467717046\ldots$ denote the unique zero of $\frac{d}{dx} \frac{\li(x)}{x}$, which is also the unique solution to the equation $\li(x) = \frac{x}{\log x}$.
Alternatively, $\log \alpha = 1.347155251069\ldots$ is the unique zero of $\frac{d}{dx} \frac{\li(e^x)}{e^x}$ and is the unique solution to the equation $x\li(e^x) = e^x$.  The function $\frac{\li(x)}{x}$ is strictly increasing on $(1,\alpha]$ and strictly decreasing on $[\alpha, \infty)$. Likewise, the function   $\frac{\li(e^x)}{e^x}$ is strictly increasing on $(0, \log \alpha]$ and strictly decreasing on $[\log \alpha, \infty)$.

\begin{proposition}\label{mono}  Let $N$ be a positive integer.  One has the following.
\begin{enumerate}
\item The function $\beta(t)-\frac{\li(e^t R_t)}{e^t}$ is strictly totally monotone on the interval $[\log \mu , \infty)$ (where $R_t = 1$).   Moreover, the function $\beta(t)$ is strictly decreasing on the interval $[\log \alpha, \infty)$, and therefore  $\beta(t) \leq \beta(\log \alpha) < 0.7695247294$ 
on $[\log \alpha, \infty)$.
\item The function $\beta(t)-\frac{\li(e^t R_t)}{e^t}$  is strictly totally monotone, and the function 
$\frac{\li(e^t R_t)}{e^t}$ is strictly increasing and concave down, on the interval $[\log (\mu/N), \log (\mu/(N-1)))$, if $N \geq 2$.
\item $\beta(t)$ is strictly increasing on the interval $[\log (\mu/N), \log (\mu/(N-1)))$ if $N \geq 3$.
\item $\beta(t)$ is strictly increasing on the interval $[\log (\mu/2), \log(\alpha/3)]$ (on which $R_t = 2$).
\end{enumerate}
\end{proposition}

\begin{proof}
We have already proved statement (1), so we may suppose that $R_t \geq 2$.  For $t \in I = [\log (\mu/N), \log (\mu/(N-1)))$, the function $R_t = N$ is constant.   Therefore, by Proposition \ref{monotone}, the function
$\beta(t) - \frac{\li(e^t R_t)}{e^t} = \theta(t,N)$ is strictly totally monotone on $I$.
Moreover, since $e^t N \leq \mu N/(N-1) \leq 2\mu <  \alpha$ on $I$, the function $\frac{\li(e^t R_t)}{e^t} = \frac{\li(e^t N)}{e^t}$ is strictly increasing on $I$.  Furthermore, the derivative
$ \frac{d}{dt} \frac{\li(e^t N)}{e^t} = \frac{N}{t+\log N}-\frac{\li(e^t N)}{e^t}$
is decreasing on $I$, so that  $\frac{\li(e^t R_t)}{e^t}$ is concave down on $I$.
Again by Proposition \ref{monotone}, one has
\begin{align*}
\beta'(t) &  =   \lim_{n \to \infty} \left(\frac{d}{dt}\frac{\li(e^t n)}{e^t} +\sum_{k = N}^{n-1}\frac{1}{(H_k-\gamma+t)^2}\right) \\
& \geq \lim_{n \to \infty} \left(\frac{d}{dt}\frac{\li(e^t n)}{e^t} +\sum_{k = N}^{n-1}\frac{1}{(t+\log(k+1))^2}\right)  \\
& \geq  \lim_{n \to \infty} \left(\frac{d}{dt}\frac{\li(e^t n)}{e^t} +\int_{N+1}^{n+1}\frac{dx}{(t+\log x)^2}\right)  \\
&=  \lim_{n \to \infty} \left(-\int_{\alpha e^{-t}}^{n}\frac{dx}{(t+\log x)^2}+\int_{N+1}^{n+1}\frac{dx}{(t+\log x)^2}\right)  \\
&=  \lim_{n \to \infty} \left(\int^{\alpha e^{-t}}_{N+1}\frac{dx}{(t+\log x)^2}+\int_{n}^{n+1}\frac{dx}{(t+\log x)^2}\right)  \\
& = \int^{\alpha e^{-t}}_{R_t+1}\frac{dx}{(t+\log x)^2}.
\end{align*}
Moreover, one has $\alpha e^{-t} > R_t+1$ provided that $R_t \geq 3$ since  $t < \log (\mu/(R_t-1)) < \log (\alpha/(R_t+1))$ if $R_t \geq 3$.  Therefore $\beta'(t) \geq \int^{\alpha e^{-t}}_{N+1}\frac{dx}{(t+\log x)^2 }> 0$ if $N \geq 3$.  Finally, if $R_t = 2$, then $\alpha e^{-t} > R_t+1 = 3$ provided that $t < \log (\alpha /3)$, so that $\beta'(t) > 0$ on $[\log (\mu/2), \log(\alpha/3))$.
\end{proof}

Thus, $\beta(t)$  is strictly increasing on $[\log (\mu/N), \log (\mu/(N-1)))$ as long as $N \geq 3$, but the cases $N = 1$ and $N = 2$ are still somewhat of a mystery, since we only know that $\beta(t)$ is strictly decreasing on $[\log \alpha, \infty) \subsetneq [\log \mu, \infty)$ and $\beta(t)$ is strictly increasing on $[\log (\mu/2), \log(\alpha/3)] \subsetneq [\log (\mu/2), \log \mu)$.  The only remaining intervals to examine, then, are $[\log  \mu, \log \alpha]$ and $[\log(\alpha/3), \log \mu)$.

Let us examine the first interval.  Since $\frac{\li(e^x)}{e^x}$ is a reasonable approximation (and lower bound) of $\beta(t)$ on $[\log \mu, \infty)$, one might expect that there exists a constant $c \geq \log \mu$ such that $\beta(t)$ is increasing on $[\log \mu, c]$ and decreasing on $[c, \infty)$.  This expectation is realized if the following two plausible conjectures hold: (1) for all positive integers $n$,  function $\beta_n(t) = \frac{ \li(ne^t)}{e^t}- \sum_{k = 1}^{n-1} \frac{1}{H_k-\gamma+t}$ has a unique local maximum  on $[\log \mu, \infty)$ at some $\rho_n \in (\log \mu,\log \alpha)$,  and (2) the limit
$\rho = \lim_{n  \to \infty} \rho_n$ exists.  (Numerical evidence leads one to suspect further that: (3) the $\rho_n$ are bounded below by $\rho$, and (4) the $\rho_n$ are strictly decreasing as $n \to \infty$, which, together with (1), would imply (2).)   Suppose, for the sake of argument, that conjectures (1) and (2) are true.   Let $\varepsilon > 0$.  Then the $\beta_n(t)$ are decreasing on $[\rho+\varepsilon,\infty)$ for sufficiently large $n$, whence $\beta(t) = \lim_{n \to \infty} \beta_n(t)$ is also decreasing on $[\rho+\varepsilon,\infty)$.   At the same time, the $\beta_n(t)$ are increasing on $[\log \mu, \rho-\varepsilon]$ for sufficiently large $n$, so that $\beta(t)$ is  increasing on $[\log \mu, \rho-\varepsilon]$.  Therefore, if conjectures (1) and (2) are true, then $\beta(t)$ is increasing on $[\log \mu, \rho]$ and decreasing on $[\rho, \infty)$, and therefore $\beta(t)$ attains a local maximum at $t =\rho$.  Table 3 lists approximate values of $\rho_n$ for $n = 1,2,3, \ldots, 10$, where $\beta_n(t)$ attains a unique local maximum at the given values of $t = \rho_n$, and also for $n = 4000$ and $n = 5000$, where $\beta_n(t)$ attains at least one local maximum at $t \approx 1.28202$.  Thus, from the computations in Table 3, it appears that $\rho \approx 1.28202$ exists.   A  separate calculation, shown in Table 4,  shows that indeed  $\beta(t)$ attains at least one local maximum value of approximately $0.77067$ at some $t$ near $1.282$.   More precisely, from the calculations in Table 4 one has
$$\beta(1.274) < 0.770653< 0.770657 < \beta(1.282)$$
and
$$ \beta(1.290) <   0.770653< 0.770657 <  \beta(1.282),$$
and therefore, since $\beta(t)$ is differentiable on $(\log \mu, \infty)$, one has the following.

\begin{table}\label{table6} 
\caption{Local maximum of $\beta_n(t)$ on $[\log \mu, \infty)$ attained at $t = \rho_n$} \bigskip
\footnotesize
\centering 
\begin{tabular}{l|l} 
$n$ & $\rho_n$ \\ \hline
$1$ &  $\log \alpha \approx  1.347155$ \\
$2$ & $\approx 1.29475$  \\
$3$ & $\approx 1.28724$   \\
$4$ & $\approx 1.28489$  \\
$5$ &  $\approx 1.28386$ \\
$6$ & $\approx 1.28331$   \\
$7$ & $\approx 1.28298$  \\
$8$ &  $\approx 1.28277$  \\
$9$ &  $\approx 1.2826260$  \\
$10$ &   $\approx 1.2825221$ \\
$4000$ & $\approx 1.282020$ \\
$5000$ & $\approx 1.282020$
\end{tabular}
\end{table}

\begin{table}\label{table3}
\caption{Upper and lower bounds of $\beta(t)$ computed with $n = 100$, and approximations of $\beta(t)$ with $n = 1000$} \bigskip
\footnotesize
\centering 
\begin{tabular}{l|l|l|l} 
$t$  & $\beta(t) <$   &  $\beta(t) \approx$  &   $\beta(t) >$   \\ \hline 
1.274 & 0.770653  &  0.770651 &  0.770639   \\
1.280  & 0.770670  & 0.770668 & 0.770656 \\
1.281  & 0.770671  &  0.770669  & 0.770657  \\
1.282  & 0.770671  & 0.770669 & 0.770657   \\
1.283  & 0.770671 & 0.770669   & 0.770657  \\
1.284  & 0.770670   & 0.770668  & 0.770656  \\
1.285  & 0.770669   &  0.770667 & 0.770655   \\
1.290  & 0.770653   & 0.770651 &  0.770663  \\
\end{tabular}
\end{table}

\begin{proposition}
The function $\beta(t)$ attains at least one local maximum value at some $t = \rho$ satisfying
$1.274 < \rho < 1.290.$
\end{proposition}

A similar analysis of $\beta(t)$ on the interval $[\log(\mu/2),\log \mu)$ suggests that $\beta(t)$ is strictly increasing on the entire interval, not just on $[\log(\mu/2), \log(\alpha/3)]$.  Thus we pose the following.

\begin{conjecture}
There exists a constant $\rho \in (\log \mu, \log \alpha)$ ($\rho \approx  1.282$) such that $\beta(t)$ is strictly increasing on $[\log \mu, \rho]$ and strictly decreasing on $[\rho, \infty)$.  Moreover,  $\beta(t)$ is strictly increasing on $[\log(\mu/2), \log \mu)$.
\end{conjecture}

\end{document}